\documentclass[amssymb,amsfonts,refcheck,12pt,verbatim,righttag]{amsart}

 \numberwithin{equation}{section}
\usepackage{amsfonts}
\usepackage{graphicx}
\usepackage{color}

\theoremstyle{plain}

\newtheorem{thm}{Theorem}[section]

\newtheorem{lem}{Lemma}
\numberwithin{lem}{section}

\newtheorem{pro}{Proposition}
\numberwithin{pro}{section}

\newtheorem{cor}{Corollary}
\numberwithin{cor}{section}

\newtheorem{ex}{Example}
\numberwithin{ex}{section}

\newtheorem{de}{Definition}
\numberwithin{de}{section}

\newtheorem{rem}{Remark}
\numberwithin{rem}{section}

\def\R {{\Bbb R}}
\def\N {{\Bbb N}}

\def\M {{\mathcal M}}

\def\Q {{\Bbb Q}}

\def\F {{\mathcal F}}

\def\A {{ A}}
\def\B {{B}}
\def\O{{O}}

\def\o{{O}}
\def\p{{A}}
\def\q{{B}}

\def\C {\pmb{C}}
\def\F{{\mathbf F}}
\def\E{{\mathbf E}}
\def\cF{{\mathcal F}}
\def\f{{\mathcal F}}

\def\ba{{\bf a}}
\def\bb{{\bf b}}
\def\bc{{\bf c}}
\def\bu{{\bf u}}
\def\bv{{\bf v}}
\def\bx{{\bf x}}
\def\by{{\bf y}}
\def\bw{{\bf w}}
\def\bz{{\bf z}}
\begin{document}
\baselineskip 18pt

\title[Self-similar subsets of the Cantor set]{Self-similar subsets of the Cantor set}

\date{}

\author{ De-Jun FENG}
\address{
Department of Mathematics,
The Chinese University of Hong Kong,
Shatin,  Hong Kong}

\email{djfeng@math.cuhk.edu.hk}

\author{Hui Rao}
\address
{Department of Mathematics, Central China Normal University, Wuhan 430070, Hubei,
Peoples's republic of China
}
\email{hrao@mail.ccnu.edu.cn}
\author{Yang Wang}
\address
{Department of Mathematics, Michigan State University, East Lansing, MI 48824, United States
}
\email{ywang@math.msu.edu}



\maketitle

\begin{abstract}
In this paper, we study the following question raised by Mattila  in 1998: what are the self-similar subsets of the middle-third Cantor set $\C$? We give criteria for a complete classification of all such subsets. We show that for any  self-similar subset $\F$ of $\C$ containing more than one point every linear generating  IFS of $\F$ must consist of similitudes with contraction ratios $\pm 3^{-n}$, $n\in \N$.
In particular, a  simple  criterion is formulated to characterize self-similar subsets of $\C$ with equal contraction ratio in modulus.
\end{abstract}

\section{Introduction}
\label{S-1}
Let $\C$ denote the standard middle-third Cantor set. The main goal of this paper is to answer the following open question raised by Mattila \cite{Cso98} in 1998:  what are the self-similar subsets of  $\C$?

 Recall that a  non-empty compact set $\F\subset\R$ is said to be  {\it self-similar} if there exists a finite family  $\Phi=\{\phi_i\}_{i=1}^k$ of contracting similarity maps on $\R$ such that \begin{equation}
 \label{e-0.1}
 \F=\bigcup_{i=1}^k\phi_i(\F).
 \end{equation}
  Such $\Phi$ is called a linear {\it iterated function system} (IFS) on $\R$. As proved by Hutchinson \cite{Hut81}, for a given IFS $\Phi$, there is a unique non-empty compact set $\F$ satisfying \eqref{e-0.1}. To specify the relation between $\Phi$ and $\F$, we call  $\Phi$  a {\it linear generating IFS} of $\F$, and $\F$  the {\it attractor} of $\Phi$. Throughout this paper, we  use $\F_\Phi$ to denote the  attractor of a given linear IFS $\Phi$.   A self-similar set $\F=\F_\Phi$ is said to be  {\it non-trivial} if it is not a singleton.

The middle-third Cantor set $\C$ is one of the most  well known examples of self-similar sets.
It has a generating IFS $\{x/3, (x+2)/3\}$.

The first result of this paper is the following theorem, which is  our starting point for further investigations.

\begin{thm}
\label{thm-1.1}
Assume that $\F\subseteq \C$ is a non-trivial self-similar set, generated by a linear IFS
$\Phi=\{\phi_i\}_{i=1}^k$ on $\R$.
Then for each $1\leq i\leq k$, $\phi_i$ has contraction ratio $\pm 3^{-m_i}$, where $m_i\in \N$.
\end{thm}

The proof of the above theorem is based on a short geometric argument  and a fundamental result of Salem and Zygmund on the sets of uniqueness in harmonic analysis.

It is easy to see that if a self-similar set $\F$ has a generating IFS $\Phi=\{\phi_i\}_{i=1}^k$
that is {\it derived} from the IFS $\Psi:=\{x/3, (x+2)/3\}$,
i.e. each map in $\Phi$ is a finite composition of maps in $\Psi$, then $\F\subseteq \C$.
In light of Theorem \ref{thm-1.1}, one may guess that each nontrivial  self-similar subset of $\C$ has a linear generating IFS derived from $\Psi$. However, this is not true.
The following counter example  was constructed in \cite{FeWa09}.

\begin{ex}
Let
$\Phi =\{\frac{1}{9}x, \frac{1}{9}(x+2)\}$.
Choose a sequence $(\epsilon_n)_{n=1}^\infty$ with $\epsilon_n\in \{0,2\}$
 so that
$w=\sum_{n=1}^\infty \epsilon_n 3^{-2n+1}$ is an irrational number.
Then by looking at the ternary expansion of the elements in $\F_\Phi+w:=\{x+w:\; x\in \F_\Phi\}$,  it is easy to see that $\F_\Phi+w\subset \C$.
Observe that $\F_\Phi+w$ is a self-similar subset of $\C$ since it is the attractor of the IFS $\Phi'=\{\frac{1}{9}(x+8w), \frac{1}{9}(x+2+8w)\}$. However any generating IFS of $\F_{\Phi'}$ can not be derived from the original
IFS $\{\psi_0=x/3,\psi_1=(x+2)/3\}$, since $w=\min \F_{\Phi'}$ can not be the fixed point of any map $\psi_{i_1i_2\ldots i_n}$ composed from  $\psi_0, \psi_1$ due to the irrationality of $w$.
\end{ex}

The above construction actually shows that  $\C$ has  uncountably many non-trivial self-similar subsets, and indicates the non-triviality of Mattila's question.

To further describe self-similar subsets of $\C$, we consider a special class of linear IFSs.
Let  ${\mathcal F}$ denote the collection of  linear IFSs $\Phi=\{\phi_i\}_{i=1}^k$ on $\R$ that satisfy the following conditions (i)-(iii):
  \begin{itemize}
  \item[(i)] $k\geq 2$ and for each $1\leq i\leq k$, $\phi_i$ is of the form
       \begin{equation}
  \label{e-0}
\phi_i(x)=s_i 3^{-m_i}x+d_i, \quad s_i=\pm 1, \ m_i\in \N,\ 0\leq d_i\leq 1.
\end{equation}
\item[(ii)]$s_1=1$ and $d_1=0$.
\item[(iii)] $d_i>0$ for at least one $1\leq i\leq k$. Moreover $d_i>0$ if  $s_i=-1$.
\end{itemize}

It is easy to check that for each $\Phi\in {\mathcal F}$, $\F_\Phi\subseteq [0,2]$ and $\min \F_\Phi=0$.  By  the symmetry of $\C$ (i.e. $\C=1-\C$),  if $\F$ is a non-trivial self-similar subset of $\C$,  then  $1-\F$ is also a self-similar subset of $\C$.
 The following result is a simple consequence of Theorem \ref{thm-1.1} (see Appendix~\ref{sec-appendix} for a proof).
 \begin{lem}
\label{lem-b}
Let $\F\subseteq \C$ be a non-trivial self-similar set.
 Then either $\F$ or $1-\F$ can be written as $a+\E$, where $a\in \C$ and $\E=\F_\Phi$ for some  $\Phi\in \f$.
\end{lem}

Hence, to characterize the self-similar subsets of $\C$, it is equivalent to characterize  the  pairs $(a, \Phi)\in  \C\times \f$ so that  $a+\F_\Phi\subseteq \C$. This is what we will carry out in this paper.

We start by introducing some notation and definitions.  Notice that any real number $x\in [0,1]$ admits at least one ternary expansion $$
x=\sum_{n=1}^\infty u_n 3^{-n}
$$
with $u_n\in \{0,2,-2\}$, and  certain rational numbers may admit two different such expansions.  For our purpose, we construct a subset $\Theta\subset \left\{0,2, \overline{2}\right\}^\N$, with convention $\overline{2}:=-2$,  by
\begin{equation}
\label{e-e6*}
\Theta=\left\{\bu\in \left\{0, 2, \overline{2}\right\}^\N:\; \iota(\bu)\neq -2,\; \bu \mbox{ does not end with } 2 \overline{2}^\infty \mbox{ and } \overline{2} 2^\infty\right\},
\end{equation}
where
\begin{equation}
\label{e-e5*}
\iota(\bu)=\left\{\begin{array}{ll}
0, & \mbox{ if } \bu= 0^\infty,\\
\mbox{the first non-zero letter appearing in $\bu$}, & \mbox{ if } \bu\neq  0^\infty.
  \end{array}
  \right.
\end{equation}
Here we say that $\bu$ ends with $\bc\in \left\{0,2,\overline{2}\right\}^\N$ if $\bu=\omega \bc$ for some finite or empty word $\omega$ over the alphabet $\left\{0, 2, \overline{2}\right\}$.  Then it is easily checked that  every $x\in [0,1]$ admits one and only one ternary expansion  $
\sum_{n=1}^\infty u_n 3^{-n}
$ with  $\bu=(u_n)_{n=1}^\infty\in \Theta$.

\begin{de}
\label{de-1.1}
For  $x\in [0,1]$, the unique ternary expansion  $
x=\sum_{n=1}^\infty u_n 3^{-n}
$ with  $(u_n)_{n=1}^\infty\in \Theta$, is called the intrinsic ternary expansion of $x$.
\end{de}

Clearly  any $x\in \C$ admits an intrinsic ternary expansion $\sum_{n=1}^\infty b_n 3^{-n}$ with $b_n\in \{0,2\}$. Now we define the notion of intrinsic translation matrix.

\begin{de} Let
$\Phi=\{\phi_i(x)=s_i3^{-m_i}x+d_i\}_{i=1}^k\in \cF$. Let $\sum_{n=1}^\infty u_{i,n} 3^{-n}$ be the intrinsic ternary expansion of $d_i$, $i=1,\ldots, k$. Denote by $U_\Phi$  the following $k\times \infty$ matrix with entries in $\left \{0, 2, \overline{2}\right\}$:
\begin{equation}
\label{e-1.6}
U_\Phi=(u_{i,n})_{1\leq i\leq k, n\geq 1}.
\end{equation}
We call $U_\Phi$ the intrinsic translation matrix of $\Phi$.
\end{de}

For  $n\in \N$, let $U_\Phi(n)$ denote the $n$-th column of $U_\Phi$, i.e.
$$U_\Phi(n)=(u_{1,n},\ldots, u_{k,n})^T,$$
 where the superscript $T$ denotes the transpose.  We say that the vector $U_\Phi(n)$ is  {\it positive}   if $u_{i,n}\in \{0,2\}$ for all $1\leq i\leq k$ and  $u_{j,n}=2$ for at least one $j\in \{1,\ldots, k\}$; correspondingly, we say that  $U_\Phi(n)$ is  {\it negative} if $u_{i,n}\in \{0,\overline{2}\}$ for all $1\leq i\leq k$ and  $u_{j,n}=\overline{2}$ for at least one $j\in \{1,\ldots, k\}$.



For $m\in \N$, let $\cF_m^+$  denote the collection of $\Phi\in \cF$ with uniform contraction ratio $3^{-m}$, and $\cF_m$ the collection of $\Phi\in \cF$ with contraction ratios $\pm 3^{-m}$.

Now we are ready to present one of our main results.

\begin{thm}
\label{thm-1.4} Let $m\in \N$ and $(a, \Phi)\in \C\times \cF_m^+$. Define $U_\Phi$ as in \eqref{e-1.6}.
Write $a=\sum_{n=1}^\infty a_n 3^{-n}$ with $a_n\in \{0,2\}$.
 Then $a+\F_\Phi\subseteq \C$ if and only if  there exists a finite non-empty set $\M\subset \N$  such that the following properties (i)-(iv) hold:\begin{itemize}
\item[(i)]
Any two numbers in $\M$
are incongruent modulo $m$;
\item[(ii)]
For each $n\in \M$,  $U_\Phi(n)$ is either positive or negative;
\item[(iii)]
For each $n\in \N\backslash \M$,  $U_\Phi(n)=(0,0,\ldots,0)^T$;
\item[(iv)] For any $n\in \M$, and any integer $t\geq 0$,
$$a_{n+mt}=\left\{\begin{array}{ll}
0,  & \mbox{ if  $U_\Phi(n)$ is  positive},\\
2,  & \mbox{ if  $U_\Phi(n)$ is  negative}.
\end{array}
\right.
$$
\end{itemize}
\end{thm}

According to the above theorems \ref{thm-1.4}, every pair $(a, \Phi)\in \C\times \cF_m^+$ satisfying  $a+\F_\Phi\subseteq \C$ can be obtained as follows.  Choose a finite set $\M\subset \N$ such that any two numbers in $\M$
are incongruent modulo $m$. Let $k\geq 2$. Construct an $k\times \infty$ matrix $U$ with entries in $\left\{0, 2, \overline{2}\right\}$ such that the first row of $U$ consists of the entries $0$ only,  all the other rows are the sequences in $\Theta$ (see \eqref{e-e6*} for the definition of $\Theta$) and moreover, $U$ (instead of $U_\Phi$) fulfils the conditions (ii)-(iii) in Theorem \ref{thm-1.4}. There are at most $2^{km}$  different such matrices $U$ once $\M$ has been fixed. For $i=1,\ldots, k$, let $d_i$ be the number in $[0,1]$ with intrinsic ternary expansion given by the $i$-th row of $U$. Choose $a=\sum_{n=1}^\infty a_n 3^{-n}\in \C$ such that  $(a_n)$ satisfies (iv) of Theorem \ref{thm-1.4}. If the cardinality of $\M$ equals $m$,  there are finitely many choices for such $a$; otherwise if the cardinality is less than $m$, there are uncountable many choices for $a$.  Let $\Phi=\{\phi_i(x)= 3^{-m}x+d_i\}_{i=1}^k$. Then $a+\F_\Phi\subseteq \C$. Below we give a simple example.

\begin{ex} Let $m=4$ and $k=4$.  Set $\M=\{2,3, 5\}$. Construct a $4\times \infty$ matrix $U$ as below:
  $$
U=\left(\begin{array}{lllllllllll}
 0 & 0 & 0 & 0 & 0 & 0  &\cdots\\
0& 0& 2 & 0 & 0 & 0  &\cdots\\
0 & 2 & 2 & 0 & \overline{2} & 0  &\cdots\\
0 & 2 & 0 & 0 & \overline{2} & 0 &\cdots\\
\end{array}
 \right),
$$
where the $n$-th column $U(n)$ of $U$ is null  for $n\geq 6$.
One can check that each row of $U$ belongs to $\Theta$, $U(n)$ is either positive  or negative for $n\in \M$, and $U(n)$ is null  for $n\in \N\backslash \M$. Thus the conditions (i)-(iii) in Theorem \ref{thm-1.4} are fulfilled for $U$. Notice that  $U=U_\Phi$ for
$$\Phi=\left\{3^{-4}x, \;3^{-4}x+.002, \; 3^{-4}x+.0220\overline{2},\; 3^{-4}x+.0200\overline{2}\right\},$$
where the translation parts in the above maps are written in its intrinsic ternary  expansions.
Let  $a=.(2002)^\infty$. Then the condition (iv) in Theorem \ref{thm-1.4} is also fulfilled. Hence
by Theorem \ref{thm-1.4}, $a+\F_\Phi\subseteq \C$. Moreover, by Theorem \ref{thm-1.4}, for $b=\sum_{n=1}^\infty b_n3^{-n}\in \C$ with $(b_n)_{n=1}^\infty\in \{0,2\}^\N$,
$$
b+\F_\Phi\subseteq \C \Longleftrightarrow b_{4n+2}=0, \; b_{4n+3}=0,\; b_{4n+5}=2 \mbox{ for each integer }n\geq 0.
$$

\end{ex}

Next we consider  more general pairs $(a, \Phi)\in \C\times (\cF_m\backslash \cF^+_m)$. Let $$\Phi=\{\phi_i(x)=s_i 3^{-m}x+d_i\}_{i=1}^k \in \f_m$$
so that there exists $1\leq k'<k$ such that $s_i=1$ for $1\leq i\leq k'$ and $s_i=-1$ for $i>k'$.  Write
\begin{equation}
\label{e-t1}
\begin{split}
&{\bf e}_0:=(\underbrace{0,\ldots, 0}_{k})^T,\\
&{\bf e}_1:=(\underbrace{0,\ldots, 0}_{k'}, \underbrace{2,\ldots, 2}_{k-k' \;})^T,
\\
&{\bf e}_2:=(\underbrace{0,\ldots, 0}_{k' \;}, \underbrace{\overline{2},\ldots, \overline{2}}_{k-k' })^T.
\end{split}
\end{equation}

\begin{thm}
\label{thm-1.7}Let $m\in \N$ and  $(a, \Phi)\in \C\times (\cF_m\backslash \cF^+_m)$. Write $a=\sum_{n=1}^\infty a_n 3^{-n}$ with $a_n\in \{0,2\}$. Define $U_\Phi$ as in \eqref{e-1.6}. Then
 $a+\F_\Phi\subseteq \C$ if and only if there exists a finite non-empty set $\M\subset \N$ so that the following properties (i)-(iv) hold:
\begin{itemize}
\item[(i)]
Any two numbers in $\M$
are incongruent modulo $m$;
\item[(ii)] For each $n\in \M$,  the $n$-th column $U_\Phi(n)$ of $U_\Phi$  is either positive or negative,  and
for any $t\geq 0$,
 $$U_\Phi(n+m(t+1))=\left\{\begin{array}{ll}  {\bf e}_1, & \mbox{  if $U_\Phi(n)$ is  positive}, \\
  {\bf e}_2, &\mbox{ if $U_\Phi(n)$ is  negative};
  \end{array}
  \right.
  $$
\item[(iii)] For each $n\in \N\backslash (\M+(m\N\cup\{0\}))$,   $U_\Phi(n)={\bf e}_0$.
\item[(iv)] For any $n\in \M$, and any integer $t\geq 0$,
$$a_{n+mt}=\left\{\begin{array}{ll}
0,  & \mbox{ if $U_\Phi(n)$ is positive},\\
2,  & \mbox{ if $U_\Phi(n)$ is  negative}.
\end{array}
\right.
$$

\end{itemize}
\end{thm}

 Similar to the remark after Theorem \ref{thm-1.4},  we see that  Theorem \ref{thm-1.7}  provides us a simple algorithm to generate all the possible pairs $(a, \Phi)\in \C\times (\cF_m\backslash \cF^+_m)$ so that  $a+\F_\Phi\subseteq \C$.   Below we give an example.

\begin{ex} Let $m=3$, $k'=2$ and $k=4$.  Set $\M=\{1,2,3\}$. Construct a $4\times \infty$ matrix $U$ as below:
  $$
U=\left(\begin{array}{lllllllllll}
 0 & 0 & 0 & 0 & 0 & 0 & 0 & 0 & 0 &\cdots\\
2& 0& 2 & 0 & 0 & 0 & 0 & 0 & 0 &\cdots\\
0 & 0 & 2 & 2 & \overline{2} & 2 & 2 & \overline{2} & 2 &\cdots\\
2 & \overline{2} & 0 & 2 & \overline{2} & 2 & 2 & \overline{2} & 2 &\cdots\\
\end{array}
 \right),
$$
where the sequence  $(U(n))_{n\geq 3}$ of the columns in $U$ is  periodic with period $3$.
Let $a=.(020)^\infty$. Then $\M$, $U$ and $a$ satisfy the conditions (i)-(iv) in Theorem \ref{thm-1.7}. Notice that  $U=U_\Phi$ for
$$\Phi=\left\{3^{-3}x, 3^{-3}x+.202, -3^{-3}x+.002(2\overline{2}2)^\infty, -3^{-3}x+.2\overline{2}0(2\overline{2}2)^\infty\right\}.$$
By Theorem \ref{thm-1.7}, $a+\F_\Phi\subseteq \C$.
\end{ex}

So far,  we have given a simple and complete characterization  of those pairs $(a, \Phi)\in \C\times \cF_m$ satisfying $a+\F_\Phi\subseteq \C$ (Theorems \ref{thm-1.4} and \ref{thm-1.7}). In Section~\ref{S-6}, we will extend this result to the pairs $(a, \Phi)\in \C\times \cF$. Indeed,
 when $\cF_m$ is replaced by $\cF$,  we will show that  if  $a+\F_\Phi\subseteq \C$, then $(U_\Phi(n))_{n\geq 1}$ is eventually periodic; furthermore, we can provide a finite algorithm to  find all  the pairs $(a, \Phi)\in \C\times \cF$ satisfying  $a+\F_\Phi\subseteq \C$, once the ratios of the maps in $\Phi$ are
 pre-given (Theorems \ref{thm-7.1} and \ref{thm-7.2}).

The paper is organized as follows. In Section~\ref{S-2}, we prove Theorem \ref{thm-1.1}. In Section~\ref{S-3}, we provide the so-called addition and subtraction principles to judge whether
$a+ b\in \C$ (or $a-b\in \C$) for given $a\in \C$ and $b\in [0,1]$, by using the intrinsic ternary expansions of $a,b$. In Section~\ref{S-4}, we use them to formulate a criterion to judge
whether $a+\F_\Phi\subseteq \C$ for given $(a, \Phi)\in \C\times \cF$, by investigating certain compatible properties of the intrinsic matrix $U_\Phi$  and the  intrinsic ternary expansion of $a$.  The proofs of our main results are given in  Sections~\ref{S-5} and  \ref{S-6}.  In  Section~\ref{S-7} we give some generalizations and remarks.  A needed known result in number theory and some elementary facts about the Cantor set $\C$ are given in  Appendix \ref{sec-appendix}.


\section{The proof of Theorem \ref{thm-1.1}}
\label{S-2}

Recall that the Cantor set $\C$ is obtained in the following way. Start with the unit interval $K_0=[0,1]$. Divide it into three equal sections and remove the open middle third. Thus we have $K_1=[0, 1/3]\cup[2/3,1]$. Then we continue inductively to obtain a sequence of closed sets $(K_n)_{n=1}^\infty$, so that  $K_n$ consists of $2^n$ closed intervals of length $3^{-n}$ obtained by removing the middle (open) one-third of each intervals in $K_{n-1}$. Finally the Cantor set is given by $\C=\bigcap_{n=0}^\infty K_n$. For $n\geq 0$, each interval in $K_{n}$ is called an {\it $n$-th basic interval}.

Before proving Theorem \ref{thm-1.1}, we shall first prove the following lemma.

\begin{lem}
\label{lem-2.2}
Assume that $\F\subseteq \C$ is a non-trivial self-similar set, generated by
a linear  IFS $\Phi=\{\phi_i\}_{i=1}^k$  on $\R$. Then we have $\log |\rho_i|/\log 3\in \Q$ for each $1\leq i\leq k$, where $\rho_i$ denotes the contraction ratio of $\phi_i$.
\end{lem}
\begin{proof}
 Fix  $\phi$ in $\Phi$.  Let $\rho$ be the contraction ratio of $\phi$. We prove that $\log |\rho|/\log 3\in \Q$.

 Let $a$ be the fixed point of $\phi$. Then
$\phi(x)=a+\rho (x-a)$.
Let $\F_{\Phi}$ denote the attractor of $\Phi$. Clearly, $a\in \F_{\Phi}$. Take $b\in \F_{\Phi}$ such that $b\neq a$. Observing  that $\phi^n(b)\in \F_{\Phi}\subseteq \C$ and $\phi^n(b)=a+\rho^n(b-a)$ for any $n\in \N$, we have
$$
a+\rho^{2n}(b-a)\in \C \mbox{ and }
1-a-\rho^{2n}(b-a)\in \C,\qquad \forall\; n\in \N.
$$
(Observe  that $\C$ is symmetric in the sense $1-\C=\C$).  Hence there exist $c\in \C$ and $d>0$ such that
$$
c+\rho^{2n}d\in \C,\quad  \forall\; n\in \N.
$$

Suppose on the contrary that $\log |\rho|/\log 3\not\in \Q$. We will derive a contradiction. Pick  $0<\epsilon<\rho^2$. Since $\log |\rho|/\log 3\not\in \Q$, the set $\{2n \log |\rho| +m \log 3:\; n, m\in \N\}$ is dense in $\R$. Therefore, we can find $n,m\in \N$ such that
$$
2n\log |\rho|+m\log 3 +{\log d}\in \left(0, \log (1+\epsilon)\right),
$$
that is,
 $\rho^{2n} 3^m d\in (1,1+\epsilon)$. We rewrite it as follows.
\begin{equation}
\label{e-2.1}
3^{-m}<\rho^{2n}d <(1+\epsilon) 3^{-m}.
\end{equation}
Since $0<\epsilon<\rho^2$, we have $1/\epsilon> \rho^{-2}$. This, combining with the first inequality in \eqref{e-2.1}, implies that there exists $p\in \N$ such that
\begin{equation}
\label{e-r11}
3^{-m}\epsilon \leq \rho^{2(n+ p)}d< 3^{-m}.
\end{equation}
Let $I$ denote the $m$-th basic interval containing the point $c$, and $J$ the $m$-th basic interval containing the point $c+\rho^{2n}d$.
Due to  \eqref{e-2.1},
$I$ and $J$ are two different basic intervals with a gap of length $3^{-m}$; then the second inequality in \eqref{e-2.1} implies that, the distance between $c$ and the right endpoint of $I$ is less than $ 3^{-m}\epsilon$.
Using this information and \eqref{e-r11}, we see that  the point $c+\rho^{2(n+ p)}d$ must be located in the gap between $I$ and $J$, contradicting the fact that $c+\rho^{2(n+ p)}d\in \C$.
\end{proof}

To prove Theorem  \ref{thm-1.1}, we need to use a  result of Salem and Zygmund \cite{SaZy55, Sal63} in the  theory of
  trigonometric series.
Let us consider a trigonometric series
$$
\sum_{n=0}^\infty (a_n\cos nx+ b_n\sin nx),
$$
where the variable $x$ is real. In one of his pioneer works,  Cantor (cf. \cite{Zyg02}) showed that if this series converges everywhere to zero, it should vanish identically, i.e., $a_n=b_n=0$ for all $n$. This work leads to the following.

\begin{de}
A subset $\E$ of the circle $[0, 2\pi)$ is called a set of uniqueness,  if any trigonometric expansion
$$
\sum_{n=0}^\infty (a_n\cos nx+ b_n\sin nx),
$$ which converges to zero for $x\in [0,2\pi)\backslash \E$,  is identically zero; that is, $a_n=b_n=0$ for all $n$.

\end{de}

It is an unsolved fundamental problem to classify all sets of uniqueness.  The following significant result is due to Salem and Zygmund.
\begin{thm}
[cf. Chapter VI of  \cite{Sal63}]
\label{thm-2.3}
 Let $\E$ be a self-similar subset of $(0, 2\pi)$  generated by an IFS
 $\{\rho x+a,\rho x+b\}$ with $0<|\rho|<1/2$, $a\neq b$.
Then a necessary and sufficient condition for $\E$ to be a set of uniqueness is that $1/|\rho|$ is a Pisot number, i.e.  an algebraic integer whose algebraic conjugates are all inside the
unit disk.
\end{thm}

 According to the above result, the Cantor set  $\C$ is a set of uniqueness.
 Of course, by definition, every subset of $\C$ is  a set of uniqueness.

\begin{proof}[Proof of Theorem  \ref{thm-1.1}]  By Lemma \ref{lem-2.2}, for each $1\leq i\leq k$,  the contraction ratio of $\phi_i$ is of the form $\pm 3^{-p_i}$, where $p_i\in \Q$.
Here we need to prove that $p_i\in \N$.
Without loss of generality, we  prove that $p_1\in \N$.
Since the attractor of $\Phi$ is not a singleton, there exists $i>1$ such that $\phi_i$ and $\phi_1$ have different fixed points.
Then the two maps $\phi_1\circ \phi_i$ and $\phi_i\circ \phi_1$ are
different and they have the same contraction ratio $\rho$.
Since the attractor generated by $\{\phi_1\circ \phi_i,\; \phi_i\circ \phi_1\}$ is a subset of $\C$,
it is a set of uniqueness. By Theorem  \ref{thm-2.3}, $1/|\rho|$ is a Pisot number. Since $1/|\rho|=3^{p_1+p_i}$ and $p_1+p_i\in \Q$, to guarantee that $3^{p_1+p_i}$ is a Pisot number,
we must have $p_1+p_i\in \N$. Similarly,  considering the maps $\phi_1\circ (\phi_i^2)$ and $(\phi_i^2)\circ \phi_1$, we also have $p_1+2p_i\in \N$. This forces $p_1\in \N$.
\end{proof}

\section{Addition and subtraction principles}
\label{S-3}

In this section we consider the following basic question:  given $a\in \C$ and $b\in [0,1]$, how can we judge whether $a+b\in \C$ or $a-b\in \C$? As an answer, using the intrinsic ternary expansions we establish the so called {\it addition and subtraction principles} illustrated by Lemma \ref{lem-admissible} and Proposition \ref{pro-4.1}.  
\begin{de}
\label{de-4.1}
{
 For $\ba=(a_n)_{n=1}^\infty\in \{0,2\}^\N$ and $\bu=(u_n)_{n=1}^\infty\in \left\{0,2, \overline{2}\right\}^\N$,  say that $(\ba, \bu)$ is { plus-admissible} if  $a_n+u_n\in \{0,2\}$  for each $n$; in this case we define
 $$\ba\oplus \bu=(a_n+u_n)_{n=1}^\infty.$$
 Similarly,   say that $(\ba,\bu)$ is {minus-admissible} if $(\ba, \overline{\bu})$ is plus-admissible, where
 \begin{equation}
 \label{u-4.2}
 \overline{\bu}=(\overline{u_n})_{n=1}^\infty
 \end{equation} with  $\overline{i}=-i$  for $i\in \{0,2,\overline{2}\}$.}
\end{de}

\begin{rem}
\label{rem-4.1}
{\rm
It is direct to check that  if $(\ba,\bu)\in \{0,2\}^\N\times \left\{0,2, \overline{2}\right\}^\N$ is plus-admissible,  then $(\ba\oplus \bu, \overline{\bu})$ is plus-admissible, i.e., $(\ba\oplus \bu, \bu)$ is minus-admissible, and  $\ba=(\ba\oplus \bu) \oplus \overline{\bu}$.
}
\end{rem}

Define  $\pi: \left\{0,2, \overline{2}\right\}^\N\to [-1,1]$ by
\begin{equation}
\label{e-e1}
\pi(\bu)=\sum_{n=1}^\infty u_n3^{-n}, \qquad \forall \, \bu=(u_n)_{n=1}^\infty.
\end{equation}

The following result directly follows from the above definitions.
\begin{lem}
\label{lem-admissible}
Let  $\ba\in \{0,2\}^\N$ and $\bu\in \left\{0,2, \overline{2}\right\}^\N$.  Define $ \overline{\bu}$ as in \eqref{u-4.2}. Then the following properties hold:
\begin{itemize}
\item[(i)]
If $(\ba,\bu)$ is plus-admissible, then  $$\pi(\ba)+\pi(\bu)=\pi (\ba\oplus \bu)\in \C.$$
\item[(ii)]
If $(\ba,\bu)$ is minus-admissible, then  $$\pi(\ba)-\pi(\bu)=\pi (\ba\oplus \overline{\bu})\in \C.$$
\end{itemize}
\end{lem}

\medskip
We remark that the converse of the above lemma is not totally true.  For example,
$\pi(02^\infty)+\pi(02^\infty)=2/3\in \C$, though $(02^\infty, 02^\infty)$ is not plus-admissible; similarly
$\pi(20^\infty)+\pi(0\bar{2}^\infty)=1/3\in \C$, though $(20^\infty, 0\bar{2}^\infty)$ is not plus-admissible.
Nevertheless, we are going to show that except some special cases, the converse of Lemma \ref{lem-admissible} is true.

Let $\sigma$ denote the left shift map on $\left\{0, 2, \overline{2}\right\}^\N$, i.e., $\sigma((u_n)_{n=1}^\infty)=(u_{n+1})_{n=1}^\infty$ for
$(u_n)_{n=1}^\infty\in \left\{0, 2, \overline{2}\right\}^\N$. Recall that $\Theta$  is defined as in \eqref{e-e6*}.
Define $\Gamma, \Gamma'\subset \{0,2\}^\N\times  \Theta$ by
\begin{equation}
\label{e-s1}
\begin{split}
\Gamma=& \{(\ba,\bu):\;\exists k\geq 0\mbox { such that } (\sigma^k\ba, \sigma^k\bu)=(02^\infty, 02^\infty) \mbox{ or } (20^\infty, 0\overline{2}^\infty)\},  \end{split}
\end{equation}
\begin{equation}\label{e-s2}
\begin{split}
\Gamma'=& \{(\ba,\bu):\; \exists k\geq 0\mbox { such that } (\sigma^k\ba, \sigma^k\bu)=(20^\infty,02^\infty) \mbox{ or } (02^\infty, 0\overline{2}^\infty)\}.
\end{split}
\end{equation}

\medskip
The main result of this section is the following.
\begin{pro}
\label{pro-4.1}
Let  ${\bf a}=(a_n)_{n=1}^\infty\in \{0,2\}^\N$ and $\bu=(u_n)_{n=1}^\infty\in \Theta$.
Then we have the following statements:
\begin{itemize}
\item[(i)] Assume that   $(\ba, \bu)\not\in \Gamma$ and $\pi(\ba)+\pi({\bu})\in \C$.
Then the pair $(\ba, \bu)$ is plus-admissible.
\item[ (ii)] Assume that $(\ba, \bu)\not\in \Gamma'$ and $\pi(\ba)-\pi(\bu)\in \C$.
Then the pair $(\ba, \bu)$ is minus-admissible.
\end{itemize}
\end{pro}

Before proving Proposition \ref{pro-4.1}, we first give a lemma.

\begin{lem}
\label{lem-4.7}
Let $n\in \N$, $c_1\ldots c_n\in \{0,2\}^n$ and $x\in \R$. Then
\begin{itemize}
\item[(1)] If $c_n=0$, then $$.c_1\ldots c_n+ 3^{-n} x\not\in \C\mbox{  if }x\in (-3,0)\cup (1,2).$$
\item[(2)] If $c_n=2$, then
$$.c_1\ldots c_n+ 3^{-n} x\not\in \C\mbox { if }x\in (-1,0) \cup (1,4).$$
\end{itemize}
\end{lem}
\begin{proof}
Notice that $.c_1\ldots c_n$ is the left endpoint of an $n$-th basic interval $I$ in defining $\C$ (cf. the first paragraph in Section \ref{S-2}). Furthermore when $c_n=0$, there is an open interval of length $\geq 3^{-n+1}$ on the left hand side of $I$ containing no points in $\C$, whilst there is an open interval
of length  $3^{-n}$ on the right hand side of $I$ containing no points in $\C$, from which (1) follows.  Similarly we have (2).
\end{proof}

\begin{proof}[Proof of Proposition \ref{pro-4.1}] We first prove (i). Assume that $(\ba, \bu)\not\in \Gamma$ and $\pi(\ba)+\pi(\bu)\in \C$.   Suppose on the contrary that $(\ba, \bu)$ is not plus-admissible. Let $n\geq 1$ be the smallest number such that
$a_n+u_n\not\in \{0,2\}$. Then either $(a_n, u_n)=(0, \overline{2})$,  or $(a_n, u_n)=(2, 2)$.  Letting $c_k=a_k+u_k$, by the minimality of $n$,
we have $c_k\in \{0,2\}$  for $k=1,\ldots, n-1$.

First consider the case when  $(a_n, u_n)=(0, \overline{2})$. Since $\bu\in \Theta$, we have $u_1\neq \overline{2}$ and thus $n>1$.    Furthermore, since $\bu$ does not end with $2 \overline{2}^\infty$ and $\overline{2}2^\infty$, we have
\begin{equation*}
\label{e-r1}
-3^{-n+1}= \sum_{k=n}^\infty (-2) 3^{-k} \leq \sum_{k=n}^\infty (a_k+u_k) 3^{-k}< (-2)3^{-n}+ \sum_{k=n+1}^\infty 4\cdot 3^{-k}= 0.
\end{equation*}
That is,
\begin{equation}
\label{e-r1}
\sum_{k=n}^\infty (a_k+u_k) 3^{-k}\in [-3^{-n+1}, 0).
\end{equation}
Notice that
$$
\pi(\ba)+\pi(\bu)=.c_1\ldots c_{n-1} + \sum_{k=n}^\infty (a_k+u_k) 3^{-k} \in \C.
$$
By \eqref{e-r1} and Lemma \ref{lem-4.7}, we  have $c_{n-1}=2$ and $\sum_{k=n}^\infty (a_k+u_k) 3^{-k}=-3^{-n+1}$, where the second equality implies that
$a_k=0$ and $u_k=\overline{2}$ for  $k\geq n$. Notice that $u_{n-1}\neq 2$ (for otherwise $\bu$ ends with $2\overline{2}^\infty$), and $a_{n-1}+u_{n-1}=c_{n-1}=2$. We have $a_{n-1}=2$ and $u_{n-1}=0$.  Therefore $\sigma^{n-2}\ba=20^\infty$ and $\sigma^{n-2}\bu= 0\overline{2}^\infty$. Hence  $(\ba, \bu)\in \Gamma$, leading to a contradiction.

Next consider the case when $(a_n, u_n)=(2,2)$. We have
$$
3^{-n+1}=4\cdot 3^{-n}+\sum_{k=n+1}^\infty (-2) 3^{-k}< \sum_{k=n}^\infty (a_k+u_k) 3^{-k}\leq  \sum_{k=n}^\infty 4\cdot 3^{-k}=2\cdot 3^{-n+1}.
$$
That is,
\begin{equation}
\label{e-r2}
\sum_{k=n}^\infty (a_k+u_k) 3^{-k}\in (3^{-n+1}, 2\cdot 3^{-n+1}].
\end{equation}
By \eqref{e-r2} and Lemma \ref{lem-4.7}, we have  $c_{n-1}=0$ and $\sum_{k=n}^\infty (a_k+u_k) 3^{-k}=2\cdot 3^{-n+1}$, where the second equality implies that
$a_k=2$ and $u_k=2$ for  $k\geq n$. Notice that $u_{n-1}\neq \overline{2}$ (for otherwise $\bu$ ends with $\overline{2} 2^\infty$), and $a_{n-1}+u_{n-1}=c_{n-1}=0$. We have $a_{n-1}=0$ and $u_{n-1}=0$.  Therefore $\sigma^{n-2}\ba=\sigma^{n-2}\bu= 0{2}^\infty$. Hence  $(\ba, \bu)\in \Gamma$, leading to a contradiction. This finishes the proof of (i).

Now  we prove (ii).  Assume that $(\ba,\bu)\not\in \Gamma'$ and $\pi(\ba)-\pi(\bu)\in \C$.
Take   $\bc\in \{0,2\}^\N$ so that $\pi(\bc)=\pi(\ba)-\pi(\bu)$. Then $\pi(\bc)+\pi(\bu)=\pi(\ba)\in \C$.
Note that $(\ba,\bu)\not\in \Gamma'$ implies $(\bc, \bu)\not\in \Gamma$. Hence by (i), $(\bc, \bu)$ is plus-admissible. By Lemma \ref{lem-admissible}, $\ba=\bc\oplus \bu$.
By Remark \ref{rem-4.1}, $(\bc\oplus \bu, \overline{\bu})$ is plus-admissible. Hence $(\ba,\overline{\bu})$ is plus-admissible, i.e., $(\ba,\bu)$ is minus-admissible. This proves (ii). \end{proof}

In the end of this section, we give one more definition.

\begin{de}
\label{de-4.3}
{\rm
In general, for $\ba\in \{0,2\}^\N$ and $\bu_1,\ldots, \bu_n\in \left\{0,2, \overline{2}\right\}^\N$,
we say that $(\ba, \bu_1, \bu_2,\ldots, \bu_n)$ is plus-admissible if
$(\ba_j, \bu_{j+1})$ is plus-admissible for  $j=0, \ldots,  n-1$, where $\ba_j$'s  are defined inductively
by   $\ba_0:=\ba$, $\ba_1:=\ba_0\oplus \bu_1$, ...,
$\ba_{n-1}=\ba_{n-2}\oplus \bu_{n-1}$.
}
\end{de}
\begin{rem}
\label{rem-4.2}
{\rm Let $\ba\in \{0,2\}^\N$ and $\bu_i=(u_{i, p})_{ p=1}^\infty\in \left\{0,2, \overline{2}\right\}^\N$, $i=1,\ldots, n$.
It is easy to see that $(\ba, \bu_1, \bu_2,\ldots, \bu_n)$ is plus-admissible if and only if
for each $ p\in \N$, the following properties hold:
\begin{itemize}
\item[(i)] Neither two consecutive letters $2$, nor  two consecutive $\overline{2}$  appear in the finite sequence (might be empty) obtained by deleting all letters $0$ from $(u_{i, p})_{i=1}^n$;
\item[(ii)] $a_p=0$ if the first letter in the above sequence is $2$, and $a_p=2$ if the first letter is $\overline{2}$.
\end{itemize}
}
\end{rem}

\section{A matrix-valued function $V$ and  matching properties}
\label{S-4}
Let $(a, \Phi)$ be a pair so that $a\in \C$ and
$\Phi=\{\phi_i\}_{i=1}^k\in \f$.
In this section we define a matrix-valued function $V$ over $\{1,\ldots, k\}^*:=\bigcup_{n\geq 1}\{1,\ldots,k\}^n$, and  show that $a+\F_\Phi\subseteq \C$ if and only if certain matching properties hold for $V$.  This criterion plays a key role in the proofs of our main results.

Suppose that
the maps $\phi_i$ are of the form $\phi_i(x)=s_i3^{-m_i}x+d_i$ as in \eqref{e-0}.
Let $a=\sum_{n=1}^\infty a_n 3^{-n}$ and $d_i=\sum_{n=1}^\infty u_{i,n} 3^{-n}$ ($i=1,\ldots, k$) be the intrinsic ternary expansions of $a, d_1,\ldots, d_k$. Write  $\ba=(a_n)_{n=1}^\infty$  and $\bu_i=(u_{i,n})_{n=1}^\infty$. Then $\ba\in \{0,2\}^\N$ and $\bu_i\in \Theta$ for $1\leq i\leq k$, where
$\Theta$ is defined as in \eqref{e-e6*}.

For $n\in \N$ and  $\bx=x_1\ldots x_n\in \{1,\ldots, k\}^n$, define an $n\times \infty$ matrix
$V(\bx):=V_\Phi(\bx)$ by
\begin{equation}
\label{e-1.e}
V(\bx)=(v_{j,p})_{1\leq j\leq n,p\geq 1},
\end{equation}
where $\bv_1=(v_{1,p})_{p=1}^\infty:=\bu_{x_1}$ and for $2\leq j\leq n$,
$$
\bv_j=(v_{j,p})_{p=1}^\infty:=\left\{\begin{array}{ll} 0^{m_{x_1}+\ldots+m_{x_{j-1}}}\bu_{x_j}, & \mbox{ if } s_{x_1}\ldots s_{x_{j-1}}=1,\\
 0^{m_{x_1}+\ldots+m_{x_{j-1}}}\overline{\bu_{x_j}}, & \mbox{ if } s_{x_1}\ldots s_{x_{j-1}}=-1.
\end{array}
\right.
$$

The above matrix-valued function $V(\cdot)$ plays an important role in our further analysis.   Take the convention that  $|\bx|=n$ if $\bx\in \{1,\ldots, k\}^n$. The following simple lemma just follows from the definition of $V(\cdot)$.

\begin{lem}
\label{lem-simple}
 Let $\bx=x_1\ldots x_i$, $\bz=z_1\ldots z_j\in \{1,\ldots, k\}^*$ and $n\in \N$. then
\begin{itemize}
\item[(i)] For $1\leq p\leq i$, the $(p,n)$-entry of $V(\bx)$ is equal to
$$
s_{x_1}\ldots s_{x_{p-1}} u_{x_p, n-\sum_{t=1}^{p-1}m_{x_t}}.
$$
\item[(ii)] If $d_{x_p}=0$, then the $(p,n)$-entry of  $V(\bx)$ is zero.
\item[(iii)] If $s_{\bz}:=s_{z_1}\ldots s_{z_j}=1$,    then the $n$-th column of $V(\bx)$ is a suffix of the $(n+q)$-th column of  $V(\bz\bx)$,
where  $q=\sum_{t=1}^jm_{z_t}$; that is, the last $|\bx|$-entries in the $(n+q)$-th column of  $V(\bz\bx)$ coincide with the entries in the $n$-th column of $V(\bx)$. Moreover if $s_{\bz}=-1$, then the $n$-th column of $V(\bx)$ is a suffix of the $(n+q)$-th column of  $(-V(\bz\bx))$.

\end{itemize}

\end{lem}

Now we are ready to  formulate our criterion.
\begin{thm}
\label{thm-1.3}
$a+\F_\Phi\subseteq \C$ if and only if  for every  $\bx\in \{1,\ldots, k\}^*$ and  $p\in \N$:  \begin{itemize}
\item[(i)] Neither two consecutive letters $2$, nor  two consecutive $\overline{2}$  appear in the finite sequence (might be empty) obtained by deleting all zero entries  from the $p$-th column of the matrix $V(\bx)$;
\item[(ii)] $a_p=0$ if the first letter in the above reduced sequence is $2$, and $a_p=2$ if the first letter is $\overline{2}$.
\end{itemize}
\end{thm}

\begin{rem} Although this criterion depends on the matching properties of $V(\bx)$ over all  $\bx\in \{1,\ldots, k\}^*$,  in the coming Sections 5-6 we will show that essentially it is enough to check  the matching properties of $V(\bx)$ for finitely many $\bx$.     
\end{rem}

  Before proving Theorem \ref{thm-1.3}, we first give several  lemmas.

\begin{lem}
\label{lem-1.2}
 $a+\F_\Phi\subseteq \C$ if and only if
$a+\phi_{x_1\ldots x_n}(0)\in \C$ for every $n\in \N$ and  $x_1\ldots x_n\in \{1,\ldots,k\}^n$,
where $\phi_{x_1\ldots x_n}:=\phi_{x_1}\circ \cdots \circ \phi_{x_n}$.
\end{lem}
\begin{proof}
Since $\Phi\in \f$, we have  $0\in \F_\Phi$. Therefore
$\F_\Phi$ is just the closure of the set $\{\phi_{x_1\ldots x_n}(0):\; x_1\ldots x_n\in \{1,\ldots,k\}^n,\; n\in \N\}$, from which the lemma follows.
\end{proof}

\begin{lem}
\label{lem-5.2}
Assume that $a+F_\Phi\subseteq \C$. Then the following properties hold:
\begin{itemize}
\item[(i)] For each $1\leq i\leq k$, $(\ba, \bu_i)$ is plus-admissible.
\item[(ii)] Let $n>1$ and   $x_1\ldots x_n\in \{1,\ldots, k\}^n$. Write $$\alpha=a+\phi_{x_1 \ldots x_{n-1}}(0),\quad
\beta=3^{-(m_{x_1}+\ldots+m_{x_{n-1}})}d_{x_n}.$$
 Let  $\alpha=\sum_{n=1}^\infty c_n3^{-n}$ and  $\beta=\sum_{n=1}^\infty u_n 3^{-n}$ be the intrinsic ternary expansions of $\alpha, \beta$. Set   ${\bc}=(c_n)_{n=1}^\infty$ and $\bu=(u_n)_{n=1}^\infty$.
Then the pair $(\bc, \bu)$ is plus-admissible if $s_{x_1}\ldots s_{x_{n-1}}=1$,
and minus-admissible if  $s_{x_1}\ldots s_{x_{n-1}}=-1$.
\end{itemize}
\end{lem}
\begin{proof}
We only prove (ii), since (i) can be viewed as the variant of (ii) corresponding to the particular case when $n=1$.

First assume that  $s_{x_1}\ldots s_{x_{n-1}}=1$.  In this case, $\alpha+\beta=a+\phi_{x_1\ldots x_n}(0)$. By Lemma \ref{lem-1.2}, we have $\alpha\in \C$ and $\alpha+\beta\in \C$. By Proposition \ref{pro-4.1}(i), to  show that $({\bf c}, \bu)$ is plus-admissible,
  it suffices to show that $({\bf c}, \bu)\not\in \Gamma$.
  Assume on the contrary, $({\bf c}, {\bu})\in \Gamma$.
  Then it is direct to check that  for $t\in \N$, \begin{equation}
  \label{e-r4}
  \pi({\bf c})+3^{-t}\pi(\bu)\not\in \C \mbox{  when $t$ is large enough}.
  \end{equation}
  However, by Lemma \ref{lem-1.2},
$a+\phi_{x_1\ldots x_{n-1} 1^p x_n}(0)\in \C$ for each $p\in \N$.
Notice that
\begin{equation*}
\begin{split}
a+\phi_{x_1\ldots x_{n-1} 1^p x_n}(0)&=a+\phi_{x_1 \ldots x_{n-1}}(0)+ s_{x_1}\ldots s_{x_{n-1}} 3^{-p m_1}3^{-(m_{x_1}+\ldots+m_{x_{n-1}})}d_{x_n}\\
&=\alpha+3^{-p m_1} \beta=\pi({\bf c})+3^{-p m_1}\pi(\bu).
\end{split}
\end{equation*}
Hence we have
$\pi({\bf c})+3^{-p m_1}\pi(\bu)\in \C$ for any $p\in \N$.
This  contradicts \eqref{e-r4}.

Next assume that $s_{x_1}\ldots s_{x_{n-1}}=-1$. A similar argument shows that $({\bf c}, \bu)$ is minus-admissible. We omit the details.
\end{proof}

\begin{proof}[Proof of Theorem \ref{thm-1.3}] We first prove the `only if' part.  Let $n\in \N$ and $\bx=x_1\ldots x_n\in \{1,\ldots, k\}^n$. Assume that the combinatoric properties (i)-(ii)  hold for $\bx$. Let $\bv_1,\ldots, \bv_n$ be the rows of $V(\bx)$. By Remark \ref{rem-4.2}, $(\ba, \bv_1, \bv_2,\ldots, \bv_n)$ is plus-admissible in the sense that $(\ba_{j-1}, \bv_j)$ is plus-admissible for each $1\leq j\leq n$, where $\ba_0:=\ba$,
 $\ba_1:=\ba_0\oplus \bv_1,\cdots$, and $\ba_{j}=\ba_{j-1}\oplus\bv_j$. Applying Lemma \ref{lem-admissible} repeatedly, we have
\begin{equation}
\label{e-tc1}
a+d_{x_1}+s_{x_1}3^{-m_{x_1}} d_{x_2}+\cdots +(s_{x_1}\cdots s_{x_{j-1}})3^{-(m_{x_1}+\ldots+m_{x_{j-1}})}d_{x_j}\in \C
\end{equation}
for $j=1,\ldots, n$.
That is, $a+\phi_{x_1\ldots x_j}(0)\in \C$ for $1\leq j\leq n$.  Letting $\bx$ vary over $\{1,\ldots,k\}^*$,  we have $a+\F_\Phi\subseteq \C$ by Lemma \ref{lem-1.2}.

Next we prove the `if' part.
Assume that $a+\F_\Phi\subseteq \C$.
Then for any $n\in \N$ and  $\bx=x_1\ldots x_n\in \{1,\ldots, k\}^n$, \eqref{e-tc1} holds for any $1\leq j\leq n$. By Lemma \ref{lem-5.2},
$(\ba, \bv_1, \bv_2,\ldots, \bv_n)$ is plus-admissible.
Hence by Remark \ref{rem-4.2},  the combinatoric properties (i)-(ii) hold.
\end{proof}

\section{The proofs of Theorems \ref{thm-1.4} and \ref{thm-1.7}}
\label{S-5}

 \begin{proof}[Proof of Theorem \ref{thm-1.4}] We first prove the `if' part of the theorem. Let  $\bx\in \{1,\ldots, k\}^*$ and $p\in \N$.  The assumptions (i)-(iii) on $U_\Phi$   guarantee that
  the $p$-th column of $V(\bx)$ (cf. \eqref{e-1.e}) contains at most one element in $\{2, \overline{2}\}$. It follows that property (i) in Theorem \ref{thm-1.3} holds.  Furthermore the assumption (iv)  gurantees that property (ii) in Theorem \ref{thm-1.3} holds.
  Since $(\bx, p)$ is arbitrarily taken from $\{1,\ldots, k\}^*\times \N$, by Theorem \ref{thm-1.3}, $a+\F_\Phi\subseteq \C$.

  Next we prove the  `only if' part.  Assume that $a+\F_\Phi\subseteq \C$.  Set
\begin{equation}
\M=\left\{n\in \N:\; U_\Phi(n)\neq \{0,0,\ldots, 0\}^T\right\}.
\end{equation}
 To prove that  (i)-(iv) in Theorem \ref{thm-1.4} hold,
 it suffices to  prove that the following two  properties on $U_\Phi=(u_{i,n})_{1\leq i\leq k, n\geq 1}$ and $\ba=(a_n)_{n=1}^\infty$  hold, as they are equivalent to  properties (i)-(iv):
  \begin{itemize}
  \item[(1)] If $u_{i,p}=2$ for some pair $(i, p)$ satisfying $1\leq i\leq k$ and $p\geq 1$, then $u_{j,p}\in \{0,2\}$, $u_{j,p+mn}=0$ and $a_{p+m(n-1)}=0$  for any $1\leq j\leq k$ and $n\geq 1$.
  \item[(2)]If $u_{i,p}=\overline{2}$  for some pair $(i, p)$ satisfying $1\leq i\leq k$ and $p\geq 1$, then $u_{j,p}\in \{0,\overline{2}\}$, $u_{j,p+mn}=0$ and $a_{p+m(n-1)}=2$  for any $1\leq j\leq k$ and $n\geq 1$.
  \end{itemize}
  Without loss of generality, we only prove (1). The proof of (2) is similar.  Assume that $u_{i,p}=2$ for some  pair $(i, p)$ satisfying $1\leq i\leq k$ and $p\geq 1$.  Let $n\geq 1$.
  Then the $(p+m(n-1))$-th column of the matrix $V(1^{n-1}i)$ is $(0,\ldots,0, 2)^T$, in which the first non-zero entry is $2$. Hence by property (ii) in Theorem \ref{thm-1.3}, $a_{p+m(n-1)}=0$.
  This also implies that $u_{j,p}\neq \overline{2}$ (and  $u_{j, p+mn}\neq \overline{2}$ as well) for any $1\leq j\leq k$;  otherwise if   $u_{j,p}= \overline{2}$ for some $j$, then a similar argument shows that $a_{p+m(n-1)}=2$, leading to a contradiction. Now notice that the $(p+mn)$-th column of the matrix  $V(j1^{n-1} i)$  is
  $(u_{j, p+mn}, 0,\ldots, 0, 2)$. Since $u_{j,p+mn}\neq \overline{2}$, we get $u_{j, p+mn}=0$ from the combinatoric property (i) in Theorem \ref{thm-1.3}. This finishes the proof of (1).
 \end{proof}

\begin{proof}[Proof of Theorem \ref{thm-1.7}]
We first prove the `only if' part of the theorem.  Assume that   $a+\F_\Phi\subseteq \C$.
Define
\begin{equation}
\label{e-M}
\begin{split}
\M=\{n\in \N:\; & U_\Phi(n)\neq \{0,0,\ldots, 0\}^T, \mbox{ but } U_\Phi(n')= \{0,0,\ldots, 0\}^T\\
&\mbox{ for  all $n'<n$ with $n'\equiv n$ (mod $m$)}\} .
\end{split}
\end{equation}
Clearly, the elements in $\M$ are incongruent modulo $m$.
To prove the desired properties (ii)-(iv) for $\M$ and  $U_\Phi$ in the theorem, it is equivalent to prove Properties 1-2 listed as below:

{\sl Property 1:   If $u_{i, p}=2$ for some pair $(i, p)\in \{1,\ldots,k\}\times \N$, then
\begin{eqnarray*}
a_{p+mn}&=&0 \quad \mbox{ for all }n\geq 0,\\
u_{j, p+mn}&=&\left\{
\begin{array}{ll}
0 \quad \mbox{or}\quad  2, & \mbox{ if } 1\leq j\leq k \mbox{ and } n=0,\\
0, & \mbox{ if } 1\leq j\leq k' \mbox{ and } n\geq 1,\\
2, & \mbox{ if } k'< j\leq k \mbox{ and } n\geq 1.\\

\end{array}
\right.
\end{eqnarray*}
 }

{\sl Property 2: If $u_{i, p}=\overline{2}$ for some  pair $(i, p)\in \{1,\ldots,k\}\times \N$, then
\begin{eqnarray*}
a_{p+mn}&=&2 \quad \mbox{ for all }n\geq 0,\\
u_{j, p+mn}&=&\left\{
\begin{array}{ll}
0\quad \mbox{or}\quad  \overline{2}, & \mbox{ if } 1\leq j\leq k \mbox{ and } n=0,\\
0, & \mbox{ if } 1\leq j\leq k' \mbox{ and } n\geq 1,\\
\overline{2}, & \mbox{ if } k'< j\leq k \mbox{ and } n\geq 1.\\
\end{array}
\right.
\end{eqnarray*}
}

\medskip

Without loss of generality, we prove Property 1 only. The proof of Property 2 is essentially identical to that of  Property 1.

First we consider the case that $u_{i, p}=2$ for some pair $(i, p)$ with $i\leq k'$ and $p\in \N$. Let $\Phi'= \{\phi_{i}\}_{i=1}^{k'}$. Then
$\Phi'\in \f_m^+$ and $a+\F_{\Phi'}\subseteq a+\F_\Phi\subseteq \C$.  Applying Theorem \ref{thm-1.4} to the pair $(a, \Phi')$,   we have $a_{p+mn}=0$ for all $n\geq 0$,  $u_{j, p}\in \{0,2\}$ for $j\leq k'$ and
$u_{j, p+mn}=0$ if $n\geq 1$, $j\leq k'$. Next suppose  $j\in (k', k]$. Notice that  the $p$-th column of
$V(j)$ is  $(u_{j,p})$, and $a_p=0$.
By Theorem \ref{thm-1.3}(ii), we get $u_{j,p}\in \{0,2\}$. Let $n\geq 1$.
Since the $(p+mn)$-th column of  $V(j1^{n-1}i)$ is
\begin{equation*}
(u_{j,p+mn}, \underbrace{0,\ldots, 0}_{n-1}, \overline{2})^T.
\end{equation*}
by Theorem \ref{thm-1.3}(i), $u_{j, p+mn}\in \{0,2\}$.  If $u_{j,p+mn}=0$, then the first non-zero letter in the above column vector  is $\overline{2}$, which forces $a_{p+mn}=2$ by  Theorem \ref{thm-1.3}(ii),
leading to a  contradiction. Hence we  have   $u_{j, p+mn}=2$.

Next we consider the case that  $u_{i, p}=2$ for some pair $(i, p)$ with $k'<i\leq k$ and $p\in \N$.
Let $n\geq 0$.  The $(p+mn)$-th column of the matrix $V(1^ni)$ is $( \underbrace{0,\ldots, 0}_{n},2)^T$,
hence by   Theorem \ref{thm-1.3}(ii),   $a_{p+mn}=0$. Let $1\leq j\leq k$.
The $(p+mn)$-th column of the matrix $V(j)$ is $(u_{j, p+mn})$. Since $a_{p+mn}=0$,
by Theorem \ref{thm-1.3}(ii), we have  $u_{j,p+mn}\in \{0, 2\}$.
Next assume $n\geq 1$. We consider the cases  $j\leq k'$ and $j>k'$ separately. First suppose that $j\leq k'$.  In this case,   the  $(p+mn)$-th column of the  $V(j1^{n-1}i)$ is
$
(u_{j, p+mn}, \underbrace{0,\ldots, 0}_{n-1}, 2)^T.
$
Hence by  Theorem \ref{thm-1.3}(i), $u_{j, p+mn}\in \{0,\overline{2}\}$, which forces $u_{j,p+mn}=0$ (since we have proved $u_{j, p+mn}\in \{0,{2}\}$). Next suppose that  $j> k'$. Since $s_j=-1$,   the  $(p+mn)$-th column of the  $V(j1^{n-1}i)$ is\begin{equation*}
(u_{j, p+mn}, \underbrace{0,\ldots, 0}_{n-1}, \overline{2})^T.
\end{equation*}
If $u_{j, p+mn}=0$, then by Theorem \ref{thm-1.3}(ii), $a_{p+mn}=2$, leads to a contradiction (since we have proved $a_{p+mn}=0$).
Hence we  have  $u_{j,p+mn}=2$. This finishes the proof of Property 1.

Next  we  prove  the `if' part of Theorem \ref{thm-1.7}. Assume that there exists a finite non-empty set $\M\subset \N$ so that properties (i)-(iv) in the theorem hold. Then it is direct to check that $\M$ is given by \eqref{e-M} and Properties 1-2 hold. Let $n\in \N$ and $\bx=x_1\ldots x_n\in \{1,\ldots, k\}^*$.
 We show below that Properties (i)-(ii) in Theorem \ref{thm-1.3} hold, from which $a+\F_\Phi\subseteq \C$ follows.

Let $p\in \N$ and consider the $p$-th column $(v_{1,p}, v_{2,p}, \ldots, v_{n, p})^T$ in the matrix
$V(\bx)$.  Assume that this vector contains at least one non-zero entry.
Let $s$ be the smallest integer so that $v_{s,n}\neq 0$.
Without loss of generality,   assume $v_{s,p}=2$.
We show below that $a_p=0$, i.e., property (ii) in Theorem \ref{thm-1.3} holds.
 According to  the definition of $V(\bx)$, we have $u_{x_i, p-(i-1)m}=0$ for $i<s$ and is not equal to $0$ for $i=s$.
 Now since $u_{x_{s}, p-(s-1)m}\neq 0$, by Properties 1-2, we have
$$u_{j, p-(i-1)m}=u_{x_{s}, p-(s-1)m}\neq 0$$ if $i<s$ and  $j>k'$.
It forces that $x_i\leq k'$ for each $i<s$. Hence $u_{x_s,p-(s-1)m}=v_{s, p}=2$.
By Properties 1-2, we have $a_{p}=0$.

Now we show that property (i) in Theorem  \ref{thm-1.3}
also holds for the vector  $$(v_{1,p}, v_{2,p}, \ldots, v_{n, p})^T,$$ i.e.,
if we delete all entries $0$ from this vector, neither two consecutive letters $2$
nor  two consecutive letters $\overline{2}$ appear in the new vector.
Assume this is not true.
Without loss of generality, assume that two consecutive letters $2$ appear.
That is, there exist $1\leq s< t\leq k$ such that $v_{s, p}=v_{t,p}=2$ and $v_{i, p}=0$ for all $s<i<t$.
Hence
 $u_{x_s, p-(s-1)m}\neq 0$,  $u_{x_t, p-(t-1)m}\neq 0$ and $u_{x_i, p-(i-1)m}=0$ for $s<i<t$.
 Since $u_{x_t, p-(t-1)m}\neq 0$, by Properties 1-2, we have
 \begin{equation}
\label{e-7.1''}
u_{j, p-(j'-1)m}=
\left\{
\begin{array}{ll}
0, & \mbox{ if $j\leq k'$ and $j'<t$},\\
u_{x_t, p-(t-1)m}, & \mbox{ if $j> k'$ and $j'<t$}.\\
\end{array}
\right.
\end{equation}
 Since $u_{x_s, p-(s-1)m}\neq 0$, by \eqref{e-7.1''},
  we have $x_s>k'$ and $u_{x_s, p-(s-1)m}=u_{x_t, p-(t-1)m}$. Furthermore,
since $u_{x_i, p-(i-1)m}=0$ for $s<i<t$, by \eqref{e-7.1''}, $x_i\leq k'$ for $s<i<t$.
Since $x_s>k'$ and $x_i\leq k'$ for $s<i<t$, we have
  $(v_{s, p},v_{t,p})$ equals $(u_{x_s, p-(s-1)m}, \overline{u_{x_t, p-(t-1)m}})$ or $(\overline{u_{x_s, p-(s-1)m}}, {u_{x_t, p-(t-1)m}})$, i.e.,
$(2,\overline{2})$ or $(\overline{2},2)$. It leads to a contradiction. This finishes the proof of the theorem.
\end{proof}

\section{General self-similar subsets of $\C$}
\label{S-6}
In this section we characterize all the  self-similar subsets of $\C$. The main results are Theorems \ref{thm-7.1} and \ref{thm-7.2}, which claim that   for a given pair $(a, \Phi)\in \C\times \cF$,  there is a finite algorithm (depending on  the contraction ratios of  maps in $\Phi$)  to judge whether $a+\F_\Phi\subset \C$.

Let $a\in \C$ and $\Phi=\{\phi_i\}_{i=1}^k\in \f$,
where the maps $\phi_i$ are of the form $\phi_i(x)=s_i3^{-m_i}x+d_i$ as in \eqref{e-0}. Throughout this section, we
set
\begin{equation}
\label{e-8.1}
\begin{split}
&\ell:=\mbox{gcd} (m_i:\; 1\leq i\leq k,\; d_i=0),\\
&r:=\mbox{ gcd} (m_1,\ldots, m_k),\\
&L:=\left(\min_{1\leq i\leq k}\frac{m_i}{r}-1\right)\left(\max_{1\leq j\leq k}\frac{m_j}{r}-1\right), \\
&D:=\max_{1\leq i\leq k}m_i, \\
\end{split}
\end{equation}
and
\begin{equation}
\label{e-8.2}
\Lambda:=\left\{\sum_{i=1}^k y_im_i:\; y_i\in \N\cup \{0\} \mbox{ for }1\leq i\leq k\right\},
\end{equation}
where $\mbox{gcd}$ means  greatest common divisor.
By Lemma \ref{lem-7.1},
\begin{equation}
\label{e-7.2'}
ry\in \Lambda\; \mbox{ for any integer }y\geq L.
\end{equation}

Let $a=\sum_{n=1}^\infty a_n 3^{-n}$ and $d_i=\sum_{n=1}^\infty u_{i,n} 3^{-n}$ ($i=1,\ldots, k$) be the intrinsic ternary expansions of $a, d_1,\ldots, d_k$. Then the intrinsic translation matrix of $\Phi$ is of the form $U_\Phi=(u_{i, n})_{1\leq i\leq k, \; n\geq 1}$.

\subsection{Necessary conditions}

In this subsection we  give some necessary conditions  so that  $a+\F_\Phi\subseteq \C$.
We begin with the following lemma.
\begin{lem}
\label{lem-7.2}
Assume that $a+\F_\Phi\subseteq \C$. Suppose that two entries $u_{i,n}$ and $u_{j,n'}$ of $U_\Phi$ satisfy that
\begin{equation}
\label{e-n1}
u_{i,n}\neq 0,\quad u_{j, n'}\neq 0\quad \mbox{and}\quad n-m_i-n'\in \Lambda.
\end{equation}
Then $(u_{i,n+tm_1})_{t\geq 0}$ is a constant sequence.
\end{lem}
\begin{proof}
We assume, without loss of generality, that $u_{i,n}=2$. (The case when $u_{i,n}=\overline{2}$ can be handled similarly.) Since $n-m_i-n'\in \Lambda$, there exist non-negative integers $y_1,y_2,\ldots, y_k$  such that $n-m_i-n'=y_1m_1+y_2m_2+\ldots+y_km_k$.

For any $t\geq 0$, since the $(n+tm_1)$-column of the matrix $V(1^{t}i)$ is $$(0,\ldots,0, u_{i,n})^T=(0,\ldots,0, 2)^T,$$
 we have $a_{n+tm_1}=0$ by  Theorem \ref{thm-1.3}(ii).

Let  $\bv:=(v_1,v_2,\ldots, v_{2+y_1+y_2+\cdots+y_k})^T$ be the $n$-th column of the matrix  $$V(i 1^{y_1}2^{y_2}\ldots k^{y_k}j).$$ Then  $v_1=u_{i,n}=2$ and $v_{2+y_1+y_2+\cdots+y_k}=u_{j,n'}$ or $\overline{u_{j,n'}}$.
By Theorem \ref{thm-1.3}(i), the first non-zero entry in the  vector $(v_2,\ldots, v_{2+y_1+y_2+\cdots+y_k})^T$
should be $\overline{2}$. Now let $t\geq 1$.
Notice that the $(n+tm_1)$-th column of the matrix  $V(i1^t1^{y_1}2^{y_2}\ldots k^{y_k}j)$ is of the form
$$\bv':=(u_{i,n+tm_1}, \underbrace{0,\ldots,0}_{t},  v_2,\ldots, v_{2+y_1+y_2+\cdots+y_k})^T,$$
where the last $(y_1+y_2+\cdots+y_k+1)$-entries of $\bv'$, i.e. $v_2,\ldots, v_{2+y_1+y_2+\cdots+y_k}$,  coincide that of  $\bv$.
Since $a_{n+tm_1}=0$, by Theorem \ref{thm-1.3}(ii), we have $u_{i, n+tm_1}\in \{0,2\}$.
If $u_{i,n+tm_1}=0$, then the first non-zero entry in the  vector $\bv'$ is $\overline{2}$, the same as that in the vector $(v_2,\ldots, v_{2+y_1+y_2+\cdots+y_k})^T$; hence by
Theorem \ref{thm-1.3}(ii), we have $a_{n+tm_1}=2$, leading to a contradiction.
Hence we must have $u_{i,n+tm_1}=2$, which completes the proof of the lemma.
\end{proof}

\begin{pro}
\label{pro-7.1}Assume that $a+\F_\Phi\subseteq \C$.
Then the sequence $(U_\Phi(n))_{n\geq 1}$ is eventually periodic with period $\ell$.
\end{pro}

\begin{proof}
Since $\ell=\mbox{gcd} (m_j:\; 1\leq j\leq k,\; d_j=0)$, it is sufficient to show that for each $j$ with $d_j=0$ (which forces that  $s_j=1$ by the definition of ${\mathcal F}$),
$$
(U_\Phi(n))_{n\geq 1} \mbox{ is eventually periodic with period $m_j$}.
$$
Without loss of generality, we prove the above statement for $j=1$.  Suppose on the contrary that this statement is false for $j=1$.
Then there exist $i\in \{1,\ldots, k\}$ and $1\leq p_0\leq m_1$ such that the sequence $(u_{i, p_0+nm_1})_{n=1}^\infty$
is not eventually periodic with period $1$. Consequently, there exist infinitely many $n$ such that $u_{i,p_0+nm_1}\neq 0$.
Hence by the pigeon hole principle, there exist two integers $n_1<n_2$ such that
$$u_{i,j+n_1m_1}\neq 0, \; u_{i, j+n_2m_1}\neq 0\; \mbox{ and } n_2\equiv n_1 ({\rm mod }\; m_i).$$
 Notice that $(p_0+n_2 m_1)-m_i-(p_0+n_1m_1)=(n_2-n_1)m_1$ is a multiple of $m_i$, and hence it belongs to $\Lambda$.
By Lemma \ref{lem-7.2},   $(u_{i, p_0+n_2m_1+tm_1})_{t\geq 0}$ is eventually periodic with period $1$,
which  contradicts the fact that $(u_{i, p_0+nm_1})_{n=1}^\infty$
is not eventually periodic with period $1$.
\end{proof}

\begin{lem}
\label{lem-7.4}
Assume that $a+\F_\Phi\subseteq \C$.
Suppose that  for some  $(i, n)\in \{1,\ldots, k\}\times \N$,
  \begin{equation}
\label{e-7.3}
u_{i,n}\neq u_{i,n+\ell}=u_{i, n+t\ell},\quad \forall\; t\geq 2.
\end{equation}
Then $U_\Phi(n')=(0,\ldots, 0)^T$ for $n'\geq 1$ satisfying $n-m_i-n'\in \Lambda$.
\end{lem}

\begin{proof} Let $n'\in \N$ so that  $n-m_i-n'\in \Lambda$.
We need to prove that  $u_{j, n'}=0$ for every $1\leq j\leq k$.

Assume on the contrary that  $u_{j, n'}\neq 0$ for some $1\leq j\leq k$.
First we have $u_{i,n}=0$;  otherwise, if $u_{i,n}\neq 0$ then \eqref{e-n1} fulfils and hence by Lemma \ref{lem-7.2},
$u_{i,n}=u_{i,n+tm_1}$ for each $t\geq 1$,  which contradicts the assumption \eqref{e-7.3} since $m_1$ is a multiple of
$\ell$.

As $u_{i,n}=0$,  by \eqref{e-7.3}, $u_{i,n+\ell}=u_{i, n+t\ell}\neq 0$ for all $t\geq 2$.
Without loss of generality assume that $u_{i, n+\ell}=2$.
Then $u_{i, n+m_1}=u_{i, n+m_1+2\ell m_i}=2$, again using the fact that $m_1$ is a  multiple of $\ell$.

Let $(t_1,\ldots, t_{2\ell+1})^T$ be the  $(n+m_1+2\ell m_i)$-th column
of the matrix $V(i^{2\ell+1})$.
Since $u_{i, n+m_1}=u_{i, n+m_1+2\ell m_i}=2$ and $s_i^{2\ell}=1$, we have $t_1=t_{2\ell+1}=2$.
By Theorem \ref{thm-1.3}, the last non-zero entry in $(t_1,\ldots, t_{2\ell})^T$ takes the value $\overline{2}$.

Since $n-m_i-n'\in \Lambda$, there exist non-negative integers $y_1,\ldots, y_k$
such that $n-m_i-n'=y_1m_1+y_2m_2+\ldots+y_km_k$.
Let $\bw$ denote the word $1^{y_1}2^{y_2}\cdots k^{y_k}$.  Let
$$
\bv=(v_1,\ldots, v_{y_1+\cdots +y_k+1})^T
$$
be  the $\left(n'+\sum_{i'=1}^k y_{i'}m_{i'}\right)$-th column
of the matrix
$V(\bw j)$.  Since $u_{j, n'}\neq 0$, we have $v_{y_1+\cdots +y_k+1}\neq 0$.

Observe that the $(n+m_1)$-th column of the matrix $V(i1\bw j)$ is of the form
  $$(2, 0, \tilde{v}_1,\ldots, \tilde{v}_{y_1+\ldots+ y_k+1})^T,$$
where $(\tilde{v}_1,\ldots, \tilde{v}_{y_1+\ldots+ y_k+1})^T=\bv$ if $s_i>0$,
and $-{\bv}$ if $s_i<0$.
By Theorem \ref{thm-1.3}(i), the first non-zero entry in  $(\tilde{v}_1,\ldots, \tilde{v}_{y_1+\ldots+ y_k+1})^T$
must be $\overline{2}$.
Next consider the $(n+m_1+2\ell m_i)$-th column of the matrix $V(i^{2\ell} 1 i \bw j)$, which is of the form
\begin{equation}
\label{e-7.4}
(t_1,\ldots, t_{2\ell}, 0,0, \tilde{v}_1,\ldots, \tilde{v}_{y_1+\ldots+ y_k+1})^T.
\end{equation}
Since  the last non-zero entry in $(t_1,\ldots, t_{2\ell})^T$ is $\overline{2}$ and the first non-zero entry in $$(\tilde{v}_1,\ldots, \tilde{v}_{y_1+\ldots+ y_k+1})^T$$ is $\overline{2}$, the vector in \eqref{e-7.4} contains two consecutive entries $\overline{2}$ after deleting all zero entries, leading to a contradiction with Theorem \ref{thm-1.3}(i). This completes the proof of  the lemma. \end{proof}

\begin{rem}
\label{rem-7.5}
The conclusion of Lemma \ref{lem-7.4} holds if we replace the assumption $n-m_i-n'\in \Lambda$ by the following (stronger) condition:
\begin{equation}
\label{e-n2}
n-n'\geq m_i+Lr\; \mbox{ and } \; n'\equiv n ({\rm mod}\; r).
\end{equation}
Indeed, according to  \eqref{e-7.2'}, the condition \eqref{e-n2} implies that  $n-m_i-n'\in \Lambda$.
\end{rem}

\begin{lem}
\label{lem-6.3}
Assume that $a+\F_\Phi\subseteq \C$.  Suppose that $u_{i,n}\neq 0$ for some  $(i, n)\in \{1,\ldots, k\}\times \N$.
Then $(a_{\tilde{n}+t\ell})_{t\geq 0}$ is a constant sequence for each  $\tilde{n}$ satisfying that
  \begin{equation}
\label{e-n6}\tilde{n}-n\geq Lr \quad \mbox{and}\quad \tilde{n}\equiv n({\rm mod}\; r).
\end{equation}

\end{lem}

\begin{proof}
  Since $\ell=\mbox{gcd} (m_j:\; 1\leq j\leq k,\; d_j=0)$, it is sufficient to show that for each $j$ with $d_j=0$ (which forces that  $s_j=1$ by the definition of ${\mathcal F}$),
$
(a_{\tilde{n}+t\ell})_{t\geq 0}$   is  periodic with period ${m_j}/{\ell}$ for each $\tilde{n}$ satisfying \eqref{e-n6}.  This is equivalent to prove that
\begin{equation}
\label{e-n10}
(a_{\tilde{n}+tm_j})_{t\geq 0}  \mbox{ is a constant sequence}
\end{equation}
for each $\tilde{n}$ satisfying \eqref{e-n6}.
Without loss of generality, below we prove \eqref{e-n10}  for $j=1$.

Fix $\tilde{n}$ so that \eqref{e-n6} fulfils.
By \eqref{e-7.2'}, there exist non-negative integers $y_1,\ldots, y_k$ such that
  $\tilde{n}-n=y_1m_1+\cdots +y_km_k$. Write
 $\bw:=1^{y_1}2^{y_2}\cdots k^{y_k}$. Let
$$
\bv=(v_1,\ldots, v_{y_1+\cdots +y_k+1})^T
$$
be  the $\tilde{n}$-th column
of the matrix
$V(\bw i)$.  Since $u_{i, n}\neq 0$, we have $v_{y_1+\cdots +y_k+1}\neq 0$. Let $v'$ be the first non-zero entry in the vector $\bv$. By
Theorem \ref{thm-1.3}(ii), $a_{\tilde{n}}$ is determined by $v'$; that is,  $a_{\tilde{n}}=0$ if $v'=2$, and $2$  otherwise.

Now let $t\geq 1$.
Notice that the $(\tilde{n}+tm_1)$-th column of the matrix  $V(1^t\bw i)$ is of the form
$$\bv_t:=(\underbrace{0,\ldots,0}_{t}, v_1, v_2,\ldots, v_{y_1+y_2+\cdots+y_k+1})^T.$$
The first non-zero entry in $\bv_t$ is  $v'$. Again, by
Theorem \ref{thm-1.3}(ii), $a_{\tilde{n}+tm_1}$ is determined by $v'$. It follows that   $a_{\tilde{n}+tm_1}=a_{\tilde{n}}$. This completes the proof of the lemma.
\end{proof}

\subsection{Case $r=1$ }
Set
\begin{equation}
\label{e-n7}
n_0=\inf \left\{n:\; U_\Phi(n)\neq \{0, \ldots, 0\}^T\right\}.
\end{equation}

\begin{thm}
\label{thm-7.1}
 Assume that $r=1$. Then   $a+\F_\Phi\subseteq \C$ if and only if the following properties hold:
 \begin{itemize}
\item[(1)] The sequence $(U_\Phi(n))_{n\geq n_0+L+D}$ is periodic with period $\ell$.
\item[(2)] The sequence
 $(a_{n})_{n\geq n_0+L}$
  is periodic with period $\ell$.
 \item[(3)] Properties (i)-(ii) in Theorem \ref{thm-1.3} hold for those pairs  $(p,\bx)\in \N\times \{1,\ldots, k\}^*$ satisfying that
 $$n_0\leq p< n_0+L+D+2\ell D
 \quad\mbox{ and } \quad |\bx|\leq  p-n_0+1.
 $$
  \end{itemize}
\end{thm}

\begin{proof}
We first prove the `only if' part of the theorem.
Assume that $a+\F_\Phi\subseteq \C$.
By Theorem \ref{thm-1.3}, (3) holds.

To show (1), suppose on the contrary that (1) is false. By Proposition \ref{pro-7.1}, the sequence of the column vectors of $U_\Phi$ is eventually periodic with period $\ell$. Hence there exist $i\in \{1,\dots, k\}$
and $n\geq  n_0+L+D$ such that
$$u_{i,n}\neq u_{i,n+\ell}=u_{i,n+t\ell},\quad \forall t\geq 2.$$
Then  we have
$U_\Phi(n_0)=(0,\cdots, 0)^T$ by Remark \ref{rem-7.5},
because $n\equiv n_0 ({\rm mod}\; 1)$ and $n-n_0\geq L+D\geq m_i+Lr$.
This  contradicts \eqref{e-n7}, and completes the proof of (1).

Next we prove (2). By the definition of $n_0$,  there exists
   $i\in \{1,\dots, k\}$ so that $u_{i,n_0}\neq 0$. Since $r=1$, by Lemma \ref{lem-6.3},  $(a_{\tilde{n}+t\ell})_{t\geq 0}$ is a constant sequence
 for each $\tilde{n}\geq n_0+L$, from which (2) follows. This completes the proof of  the `only if' part of the theorem.

Now we turn to the proof of the `if' part of the theorem.
Suppose that $a+\F_\Phi$ is not a subset of $\C$.
  By Theorem \ref{thm-1.3}, at least one of the following two scenarios occurs:
 \begin{itemize}
 \item[(a)] There exist $p\in \N$ and  $\bx\in \{1,\ldots, k\}^*$
 such that the $p$-th column of $V(\bx)$ is of the form
 $(2, \underbrace{0,\ldots, 0}_{|\bx|-2},2)^T$ or $(\overline{2}, \underbrace{0,\ldots, 0}_{|\bx|-2},\overline{2})^T$.
 \medskip
 \item[(b)] There exist $p\in \N$ and  $\bx\in \{1,\ldots, k\}^*$  such that the $p$-th column of $V(\bx)$ is of the form
 $(\underbrace{0,\ldots, 0}_{|\bx|-1},u_p)^T$ with $u_p\neq 0$ and
  $(a_p, u_p)$ is not plus-admissible.
 \end{itemize}

First suppose that (a) occurs for some pair $(p, \bx)$. We may assume that $p$ is the smallest in the sense that if  (a) happens for another pair $(p', \bx')$, then $p\leq p'$.  and $\bx=x_1\ldots x_t\in \{1,\ldots, k\}^*$.    Then $$u_{x_1,p}\neq 0\; \mbox{ and }\; u_{x_t, p-v}\neq 0$$
with $v:=m_{x_1}+\ldots+ m_{x_{t-1}}$. The second condition and \eqref{e-n7} imply that $p-(m_{x_1}+\ldots+ m_{x_{t-1}})\geq n_0$ and thus
$$|\bx|=t\leq m_{x_1}+\ldots+ m_{x_{t-1}}+1\leq p-n_0+1.$$
By (3), we have
\begin{equation}
\label{e-n11}
p\geq n_0+L+D+2\ell D.
\end{equation}

Below we show that there exist $p-2\ell D\leq  q<p$ with $q\equiv p({\rm mod}\; \ell)$ and ${\bz}\in \{1,\ldots, k\}^*$ so that   (a) happens for the pair $(q, \bz)$ instead of $(p, \bx)$, leading to a contradiction with the minimality of $p$.

We assume, without loss of generality,   that  the $p$-th column of
 $V({\bx})$ is of the form $(2, \underbrace{0, \ldots, 0}_{|\bx|-2}, 2)^T$.  If $|\bx|\leq 2\ell$, then by \eqref{e-n11},
 $$p-\ell-(m_{x_1}+\ldots+m_{x_{t-1}})\geq n_0+L+D;$$
hence by (1), the $(p-\ell)$-column of $V(\bx)$ coincides with its $p$-th column. Thus  (a) happens for the pair $(p-\ell, \bx)$. Next consider the case that
$|\bx|\geq 2\ell+1$.
For  $1\leq i\leq 2\ell+1$, the pair
 $(s_{x_1}\ldots s_{x_i}, m_{x_1}+\ldots+m_{x_{i}}({\rm mod}\; \ell))$ take values in $\{1,-1\}\times \{0, 1,\ldots, \ell-1\}$, a set with cardinality $2\ell$.
 By the pigeon hole principle, there exist $1\leq i<j\leq 2\ell+1$ such that
 $$
 s_{x_1}\ldots s_{x_i}=s_{x_1}\ldots s_{x_j}\mbox{ and }m_{x_1}+\ldots+m_{x_{i}} \equiv m_{x_1}+\ldots+m_{x_j} ({\rm mod}\; \ell).
 $$
Set $\bz=x_1\ldots x_i x_{j+1}\ldots x_t$, and $q=p-(m_{x_{i+1}}+\ldots+m_{x_{j}})$. Then $q\geq p-2\ell D$ and $q\equiv p({\rm mod}\; \ell)$.
Since
\begin{equation*}
\begin{split}
q-(m_{x_1}+\ldots+m_{x_{i-1}})&=p-(m_{x_1}+\ldots+m_{x_{j}})+m_{x_i}\\
&\geq  p-2\ell D\geq n_0+D+\ell,
\end{split}
\end{equation*}
by the $\ell$-periodicity of $(U_\Phi(n))_{n\geq n_0+D+\ell}$,  the first $i$ entries in the $q$-th column of $V(\bz)$ are the same as that in the $p$-th column of $V(\bx)$. In the mean time, since $s_{x_1}\ldots s_{x_i}=s_{x_1}\ldots s_{x_j}$,
applying Lemma \ref{lem-simple}(iii) we see that the last $(t-j)$ entries in the $q$-th column of $V(\bz)$ coincide with that in the $p$-th column of
$V(\bx)$. Hence the $q$-th column of $V(\bz)$ is of the form $(2,\underbrace{0,\ldots, 0}_{|\bz|-2}, 2\}^T$. This proves the above claim, and hence derives a contradiction if  (a) occurs.

Next suppose that (b) occurs for some pair $(p, \bx)$. Again we assume that $p$ is the smallest. Following essentially the same argument as in the
above paragraph, we can find  $p-2\ell D\leq  q<p$ with $q\equiv p({\rm mod}\; \ell)$ and ${\bz}\in \{1,\ldots, k\}^*$, such that
the $q$-th column of $V(\bz)$ is of the form $(\underbrace{0,\ldots, 0}_{|\bz|-1}, u_p\}^T$. Since $q\geq p-2\ell D\geq n_0+L+D$, by (2) we have $a_{q}=a_p$, and thus $(a_q, u_p)$ is not plus admissible. Therefore  (b) occurs for the pair $(q, \bz)$, contradicting the minimality of $p$.
This completes the proof of Theorem \ref{thm-7.1}.
\end{proof}

\subsection{Case $r>1$.} Analogous to Theorem \ref{thm-7.1}, we have


\begin{thm}
\label{thm-7.2}
 Assume that $r>1$.
 For $i=0,1, \ldots, r-1$, set
 $$
 n_i=\inf\left\{n\in \N: \; U_\Phi(rn-i)\neq \{0, \ldots, 0\}^T\right\},
 $$
 with convention $\inf \emptyset=\infty$.
 Then   $a+\F_\Phi\subseteq \C$ if and only if the following properties hold for each $i\in \{0,1,\ldots, r-1\}$ with $n_i\neq \infty$:
 \begin{itemize}
\item[(1)] The sequence $(U_\Phi(rn-i))_{n\geq n_i+L+D/r}$ is periodic with period $\ell/r$.
\item[(2)] The sequence
 $(a_{rn-i})_{n\geq n_i+L}$
  is periodic with period $\ell/r$.
 \item[(3)] Properties (i)-(ii) in Theorem \ref{thm-1.3} hold for those pairs  $(rp-i,\bx)\in \N\times \{1,\ldots, k\}^*$ satisfying that
 $$n_i\leq p< n_i+L+D/r+2\ell D/r
 \quad\mbox{ and } \quad |\bx|\leq  p-n_i+1.
 $$
  \end{itemize}
\end{thm}

The proof of Theorem \ref{thm-7.2} is essentially identical to that of Theorem \ref{thm-7.1} and so is omitted.

\section{Generalizations and remarks}
\label{S-7}
In this section we give some generalizations of  our main results.
For $0<\alpha<1/2$, write
$$
\C_\alpha=\left\{ \sum_{n=1}^\infty \epsilon_n\alpha^n:\; \epsilon_n=0 \mbox{ or } 2\right\},
$$
and call $\C_\alpha$  {\it the central $\alpha$-Cantor set} \footnote{Notice that the diameter of $\C_\alpha$ is $2\alpha/(1-\alpha)$.}.  The following addition and subtraction principles hold for $0<\alpha<1/3$.

\begin{lem}
\label{lem-8.1}
Let $0<\alpha<1/3$. Let $a=\sum_{n=1}^\infty a_n \alpha^n$ with $a_n\in \{0,2\}$ and
$b=\sum_{n=1}^\infty u_n \alpha^n$ with $u_n\in \{0,2,-2\}$.  Then
$a+b\in \C_\alpha$ if and only if $a_n+b_n\in \{0,2\}$ for all $n\in \N$. Similarly,  $a-b\in \C_\alpha$ if and only if $a_n-b_n\in \{0,2\}$ for all $n\in \N$.
\end{lem}
\begin{proof}
Without loss of generality we only prove that $a+b\in \C_\alpha$ if and only if $a_n+b_n\in \{0,2\}$ for all $n\in \N$. The `if' part is trivial. To show the `only if' part,
suppose that $a+b\in \C_\alpha$. Then there exists $(c_n)\in \{0,2\}^\N$ such that
$$\sum_{n=1}^\infty(a_n+b_n)\alpha^n=\sum_{n=1}^\infty c_n \alpha^n.
$$
It is sufficient to show that $a_n+b_n=c_n$ for all $n\in \N$.  Suppose this is not true. Let $n_0$ be the smallest positive integer so that
$a_{n_0}+b_{n_0}\neq c_{n_0}$.  Then we have
\begin{equation}\label{e-7.1a}
a_{n_0}+b_{n_0}-c_{n_0}=\sum_{m=1}^\infty (c_{n_0+m}-a_{n_0+m}-b_{n_0+m}) \alpha^m.
\end{equation}
Notice that  $|a_{n_0}+b_{n_0}-c_{n_0}|\geq 2$, whilst $-4\leq c_{n_0+m}-a_{n_0+m}-b_{n_0+m}\leq 4$ for each $m\in \N$. Hence by \eqref{e-7.1a},
we have $2\leq 4\sum_{n=1}^\infty\alpha^n=4\alpha/1-\alpha$, which contradicts the assumption that $\alpha<1/3$.
\end{proof}

Due to  the above lemma,  slightly modified versions  of our main results (Theorems \ref{thm-1.4}, \ref{thm-1.7}, \ref{thm-7.1} and \ref{thm-7.2}) remain valid for $\C_\alpha$ ($0<\alpha<1/3$). The proofs are essentially identical to that for the case $\alpha=1/3$.  To be concise, below we only formulate  the modified version of Theorem \ref{thm-1.4}.

\begin{thm}
\label{thm-8.1} Let $\alpha\in (0, 1/3)$ and  $m\in \N$. Let $a=\sum_{n=1}^\infty a_n\alpha^n$ with $a_n\in \{0,2\}$ and
 $\Phi=\{\alpha^m x+d_i\}_{i=1}^k$ with $d_i=\sum_{i=1}^k u_{i,n}\alpha^n$, where
 $u_{1,n}=0$ and $u_{i,n}\in \{0,2,-2\}$ for all $n\in \N$ and $1\leq i\leq k$. Set
 $$
 U_{\Phi, \alpha}=(u_{i,n})_{1\leq i\leq k, n\geq 1},
 $$
 and let $U_{\Phi,\alpha}(n)$ denote the $n$-th column of $U_{\Phi,\alpha}$.
 Then $a+\F_\Phi\subseteq \C_\alpha$ if and only if  there exists a finite non-empty set $\M\subset \N$  such that the following properties (i)-(iv) hold:\begin{itemize}
\item[(i)]
Any two numbers in $\M$
are incongruent modulo $m$;
\item[(ii)]
For each $n\in \M$,  $U_{\Phi,\alpha}(n)$ is either positive or negative;
\item[(iii)]
For each $n\in \N\backslash \M$,  $U_{\Phi,\alpha}(n)=(0,0,\ldots,0)^T$;
\item[(iv)] For any $n\in \M$, and any integer $t\geq 0$,
$$a_{n+mt}=\left\{\begin{array}{ll}
0,  & \mbox{ if  $U_{\Phi,\alpha}(n)$ is  positive},\\
2,  & \mbox{ if  $U_{\Phi,\alpha}(n)$ is  negative}.
\end{array}
\right.
$$
\end{itemize}
\end{thm}

Now it arises a natural question whether Theorem \ref{thm-1.1} can be extended to other values $\alpha\in (0,1/3)$.
After the completion of the first version of this paper, the following extension of Theorem \ref{thm-1.1} was given  in \cite{FHR14} in a slightly variant version.
\begin{thm}\cite[Theorem 1.4]{FHR14}
\label{thm-7.3}
Let $0<\alpha<1/2$. Assume that $\F\subseteq \C_\alpha$ is a non-trivial self-similar set, generated by a linear IFS
$\Phi=\{\phi_i\}_{i=1}^k$ on $\R$.
\begin{itemize}
\item[(1)]If $\alpha\leq 1/4$, then for each $1\leq i\leq k$, $\phi_i$ has contraction ratio $\pm \alpha^{-m_i}$, where $m_i\in \N$.
\item[(2)]If $\alpha<\sqrt{2}-1$ or $1/\alpha$ is a Pisot number, then $\phi_i$ has contraction ratio $\pm \alpha^{t_i}$, where $t_i\in \Q_+$.
\end{itemize}
\end{thm}

Recall that $\lambda>1$ is called a {\it Pisot number} if it is an algebraic integer whose algebraic conjugates are all less than $1$ in modulus.

Furthermore it was observed in \cite[Lemma 4.1]{FHR14} that
there exist  countably many points  $\alpha$ in $((3-\sqrt{5})/2, 1/2)$,  namely the positive roots of $x-(x+x^2+\ldots+x^k)^2$ ($k=2,3,\ldots$), such that  $\C_\alpha$  contains a   self-similar subset $E$  which admits a contraction ratio $\alpha^{q/2}$ for an odd integer $q$ in one of the maps in a generating IFS of $E$. Nevertheless, $(3-\sqrt{5})/2\approx 0.382>1/3$.

For $p\geq 2$, let $\Lambda_p$ denote the set of $\alpha\in (0,1/2)$ so that $\C_\alpha$ contains a  self-similar subset $E$  which admits a contraction ratio $\alpha^{q/p}$, with  $q$ coprime to $p$,  in one of the maps in a generating IFS of $E$.
By Theorems \ref{thm-8.1} and \ref{thm-7.3}, we can formulate a criterion for $\alpha\in (1/4, 1/3)\cap \Lambda_p$; and in particular we can characterize  the elements in the set  $(1/4, 1/3)\cap \Lambda_2$.

\begin{thm}
\label{thm-8.2}
 Let $p\in \N$ with $p\geq 2$. Then $\alpha\in (1/4,1/3)\cap \Lambda_p$ if and only if there exists $b>0$ such that $b\alpha^{n/p}$ ($n=0,\ldots, p-1$) are of the forms
\begin{equation}
\label{e-8.1}
b \alpha^{n/p}=\sum_{k=1}^{t_n} \epsilon_{k,n} \alpha^k
\end{equation}
simultaneously, where $t_n\in \N$ and $\epsilon_{k, n}\in \{0, 2,-2\}$ for $1\leq n\leq p-1$ and $1\leq k\leq t_n$. In particular, \begin{equation*}
\begin{split}
(1/4,1/3)\cap \Lambda_2:=& \left\{ 0<\alpha<1/3:  \;  \exists \mbox{  polynomials $A(x)$ and $B(x)$ }\right.\\
\mbox{} & \left. \mbox{ of coefficients $0,\pm 2$ such that $\alpha^{1/2}=A(\alpha)/B(\alpha)$}\right\}.
\end{split}
\end{equation*}
\end{thm}
\begin{proof}
We first prove the necessity part of the theorem. Suppose that $\alpha \in (0,1/3)\cap  \Lambda_p$. It is not hard to show that there exist $a, c>0$, $m'\in \N$ such that
$\C_\alpha\supset a+\F_\Psi$, where $\Psi=\{\alpha^{m'+1/p}x, \alpha^{m'+1/p}x+c\}$. Notice that
$$
\Phi:=\Psi^p=\left\{\alpha^{pm'+1}x+c\sum_{n=0}^{p-1} \epsilon_n \alpha^{m'n+n/p}:\; \epsilon_n\in \{0,  1\} \mbox{ for $0\leq n\leq p-1$}\right\},
$$
and $\F_\Phi=\F_{\Psi}$. We have $$\C_\alpha\supset a+\F_{\Phi}.$$
Set $m=pm'+1$. By Theorem \eqref{thm-8.1},  for each $(\epsilon_0,\ldots, \epsilon_{p-1})\in \{0,1\}^p$, $c \sum_{n=0}^{p-1} \epsilon_n \alpha^{m'n+n/p}$ can be expanded into a polynomials of $\alpha$ with coefficients in $\{0, 2,-2\}$. In particular, for each $0\leq n\leq p-1$,
$c\alpha^{m'n+n/p}$ can be expanded into a polynomials of $\alpha$ with coefficients in $\{0, 2,-2\}$, which   implies that \eqref{e-8.1} holds for $b=c \alpha^{m'p}$.

Next we prove the sufficiency part. Suppose that \eqref{e-8.1} holds. Take an  integer $m'$ so that
\begin{equation}
\label{e-8.2}
m'\geq \max \{t_n:\; n=0,\ldots, p-1\}.
\end{equation}
 Let
 $$\Phi=\left\{\alpha^{pm'+1}x+b\sum_{n=0}^{p-1} \epsilon_n \alpha^{m'n+n/p}:\; \epsilon_n\in \{0,  1\} \mbox{ for $0\leq n\leq p-1$}\right\}.
$$
Denote
$$
G:=\left\{b\sum_{n=0}^{p-1} \epsilon_n \alpha^{m'n+n/p}:\; \epsilon_n\in \{0,  1\} \mbox{ for $0\leq n\leq p-1$}\right\}.
$$
The assumption \eqref{e-8.1} and the condition \eqref{e-8.2} ensure that each $g\in G$ has an expansion $\sum_{k=0}^{m'p}\epsilon_{k,g} \alpha^k$, with $\epsilon_{k,g}\in \{0,2,-2\}$; and moreover, for each $0\leq k\leq m'p$, either $\epsilon_{k,g}\in \{0,2\}$  for all   $g\in G$, or $\epsilon_{k,g}\in \{0,-2\}$  for all   $g\in G$.  Rewrite $\Phi$ as $\{\alpha^{m}x+b_i\}_{1\leq i\leq 2^p}$, with $m=pm'+1$ and $b_1=0$.  Then the matrix $U_{\Phi,\alpha}$ satisfies the conditions (i)-(iii) of Theorem \eqref{e-8.1}, with $\M\subseteq \{1, \ldots, m\}$.  Choose $(a_n)\in \{0,2\}^\N$ such that (iv) of Theorem \ref{thm-8.1} holds. Then by Theorem \eqref{e-8.1}, $a+\F_\Phi\subseteq \C_\alpha$. This completes  the proof of the sufficiency part, since  $\F_\Phi=\F_\Psi$ with $\Psi=\{\alpha^{m'+1/p}x, \alpha^{m'+1/p}x+b\}$.
\end{proof}

\begin{rem}
\begin{itemize}
\item[(i)] $1/4$ is an accumulation point of the set $(1/4,1/3)\cap \Lambda_2$, since this set contains the positive roots of $x^{1/2}=\frac{1-x-x^2-\cdots-x^n}{1+x+x^2+\cdots+ x^n}$ ($n\geq 2$).
\item[(ii)] By Theorem \ref{thm-8.2}, for each $p\geq 2$,  $(1/4,1/3)\cap \Lambda_p$ does not contain any transcendental numbers. We do not know whether $(1/4,1/3)\cap \Lambda_p
\neq \emptyset$ for some $p\geq 3$.
\end{itemize}
\end{rem}

\begin{rem}
  As our main results (Theorems \ref{thm-1.4}, \ref{thm-1.7}, \ref{thm-7.1} and \ref{thm-7.2}) can be extended to $\C_\alpha$ for $0<\alpha<1/3$, by Theorem \ref{thm-8.2} and the definition of $\Lambda_p$, we can characterize all the self-similar subsets of $\C_\alpha$, in a way similar to that for $\C$,  for all those  $\alpha$ in the following set
  $$(0,  1/4] \cup \Big[(1/4, 1/3)\backslash (\bigcup_{p\geq 2} \Lambda_p)\Big].$$

\end{rem}
\appendix

\section{ }
\label{sec-appendix}
In this part we give some known or simple results that are needed in the previous sections. We begin with the following.

\begin{lem}\cite[Theorem 1]{Bra42}
\label{lem-7.1}
Let $p_1,\ldots, p_n\in \N$ and   $h=\gcd(p_1,\ldots,p_n)$. 
Then for any
integer $\displaystyle y\geq \left(\min_{1\leq i\leq n}\frac{p_i}{h}-1\right)\left(\max_{1\leq i\leq n}\frac{p_i}{h}-1\right)$,
there exist $y_1,\ldots, y_n\in \N\cup\{0\}$ such that $$yh=\sum_{i=1}^n y_ip_i.$$
\end{lem}

The next result is elementary.
\begin{lem}
\label{lem-a}
Suppose  $\E\subseteq [0,1]$ with $0\in \E$ is a non-trivial self-similar set generated by an IFS
$\Phi=\{\phi_i\}_{i=1}^k$ of the form
$$\phi_i(x)=s_i3^{-m_i}x+d_i,\quad i=1,\ldots, k,$$
where $s_i=\pm 1,m_i\in \N$ and $d_i\in \R$.  Then $0\leq d_i\leq 1$ for all $1\leq i\leq k$, and  $d_i>0$ whenever $s_i=-1$. Moreover, there exists at least one $1\leq i\leq k$ so that $d_i>0$.
\end{lem}
\begin{proof}
Since $0\in \E$, we have $d_i=\phi_i(0)\in \E\subseteq [0,1]$ for each $1\leq i\leq k$.
Now suppose $s_i=-1$ for some $i$. Then $\E\supset \phi_i(\E)=d_i-3^{-m_i}\E$. Since $\E\subseteq [0,1]$,
we must have $d_i>0$, for otherwise, $\phi_i(\E)\subseteq (-\infty,0]$,
which is impossible because $\phi_i(\E)$ contains at least two points and $\phi_i(\E)\subset \E\subseteq [0,1]$.

Next we show that there exists at least one $i\in \{1,\ldots, k\}$ so that $d_i>0$. If this is not true, then by the above argument, $s_i=1$ and $d_i= 0$ for all $i$, which implies that $\E=\{0\}$, leading to a contradiction.
\end{proof}

Now we are ready to prove the following result (Lemma \ref{lem-b} in Section \ref{S-1}).
\begin{lem}
Let $\F\subseteq \C$ be a non-trivial self-similar set. Write $\F'=\{1-x:\; x\in \F\}$.
 Then either $\F$ or $\F'$ can be written as $a+\E$, where $a\in \C$ and $\E=\F_\Phi$ for some  $\Phi\in \f$,
 where $\F_\Phi$ denotes the attractor of $\Phi$.
\end{lem}
\begin{proof}
Let $\{S_i\}_{i=1}^k$ be a linear generating IFS of $\F$. By Theorem \ref{thm-1.1}, this IFS must be of the form
$$S_i(x)=s_i3^{-m_i}x+d_i,\quad i=1,\ldots, k,$$
where $s_i=\pm 1,m_i\in \N$ and $d_i\in \R$.    Let $[a,b]$ be the convex hull of $\F$. Then $\bigcup_{i=1}^{k}S_i([a,b])\subseteq [a,b]$, and  $\{a,b\}\subset \bigcup_{i=1}^{k}S_i([a,b])$. Without loss of generality, we assume that $a\in S_1([a,b])$ and $b\in S_k([a,b])$. There are 3 possibilities: (1) $s_1=1$;
(2) $s_k=1$; (3) $s_1=s_k=-1$.

First assume that (1) occurs. Then $S_1(a)=a$. Let $\E=\F-a:=\{x-a:\; x\in \F\}$.  Then $\E\subseteq [0,1]$ and
$$
\E=\F-a=\bigcup_{i=1}^k (S_i(\F)-a)=\bigcup_{i=1}^k (S_i(\F-a)+S_i(a)-a)=\bigcup_{i=1}^k\phi_i(\E),
$$
where the maps $\phi_i$ are defined by $\phi_i(x)=S_i(x)+S_i(a)-a$. Note that $\phi_1(x)=3^{-m_1}x$.
This together with Lemma \ref{lem-a} yields  $\Phi:=\{\phi_i\}_{i=1}^k\in \f$. Now $\F=a+\F_\Phi$.

Next assume that (2) occurs. Then $S_k(b)=b$. Let $\E=b-\F$.   Similarly, one has $\E\subseteq [0,1]$ and
$$\E=\bigcup_{i=1}^k \phi_i(\E),$$
where $\phi_{k+1-i}(x):=S_i(x)+b-S_i(b)$ for $i=1,\ldots k$. Again, we have  $\phi_1(x)=3^{-m_k}x$. Thus by Lemma \ref{lem-a} we have  $\Phi:=\{\phi_i\}_{i=1}^k\in \F$. Now  $\F'=1-\F= (1-b)+\F_\Phi$.

In the end, assume that (3) occurs. Then $a=S_1(b)$ and $b=S_k(a)$.
Hence $S_1\circ S_k(a)=a$. Let $\E=\F-a$. Then
$\E=\bigcup_{i=1}^{k}\bigcup_{i=1}^k \widetilde{\phi_{i,j}}(\E)$, where $\widetilde{\phi_{i,j}}(x)=S_i\circ S_j(x)+S_i\circ S_j(a)-a$.
Recoding the IFS $\{\widetilde{\phi_{i,j}}: 1\leq i,j\leq k\}$ to  $\{\phi_i\}_{i=1}^{k^2}$ so that $\phi_1=\widetilde\phi_{1,k}$, then $\Phi:=\{\phi_i\}_{i=1}^{k^2}\in \f$ and $\F=a+\F_\Phi$.
\end{proof}

\noindent{\bf Acknowledgements}. The authors would like to thank Ranran Huang and Ying Xiong for helpful comments.  The first author was partially supported by the RGC grant in CUHK. The second author was partially supported by NNSF (11225105, 11371339).
 The third author was supported in part by NSF Grant DMS-1043034 and DMS-0936830.

\end{document}